\documentclass[12pt]{amsart}
\usepackage{geometry}   
\usepackage[colorlinks,citecolor = red, linkcolor=blue,hyperindex]{hyperref}
\usepackage{euscript,eufrak,verbatim, mathrsfs}
\usepackage[psamsfonts]{amssymb}
\usepackage{bbm}
\usepackage{graphicx}
 \usepackage{float}
\usepackage{float, tikz}

 \usepackage[all, cmtip]{xy}

\usepackage{upref, xcolor, dsfont}
\usepackage{amsfonts,amsmath,amstext,amsbsy, amsopn,amsthm}
\usepackage{enumerate}

\usepackage{url}

\usepackage{mathtools}
\usepackage{bookmark}

 \usepackage{euscript}
\usepackage{helvet}         
\usepackage{courier}        
\usepackage{type1cm}        
\usepackage{multicol}        
\usepackage[bottom]{footmisc}

\newtheorem{theorem}{Theorem}[section]
\newtheorem*{theorem*}{Theorem B} 
\newtheorem{lemma}[theorem]{Lemma}

\newtheorem{proposition}[theorem]{Proposition}
\newtheorem{corollary}[theorem]{Corollary}
\newtheorem*{definition*}{Definition}
\newtheorem*{remark*}{Remark}

\newtheorem*{observation*}{Observation}

\newtheorem*{assumption*}{Assumption}

\theoremstyle{definition}
\newtheorem{definition}{Definition}[section]
\newtheorem{question}{Question}
\newtheorem{problem}{Problem}

\theoremstyle{remark}
\newtheorem{remark}{Remark}[section]

\newcommand{\R}{\mathbb{R}}
\newcommand{\N}{\mathbb{N}}
\newcommand{\Z}{\mathbb{Z}}

\newcommand{\D}{\mathbb{D}}
\newcommand{\C}{\mathbb{C}}
\newcommand{\E}{\mathbb{E}}
\newcommand{\T}{\mathbb{T}}

\newcommand{\B}{\mathbb{B}}
\newcommand{\BS}{\mathbb{S}}
\newcommand{\F}{\mathbb{F}}

\newcommand{\TT}{\mathcal{T}}

\newcommand{\HH}{\mathcal{H}}
\newcommand{\BH}{\bold{H}}

\newcommand{\spann}{\mathrm{span}}

\newcommand{\Cov}{\mathrm{Cov}}

\newcommand{\pless}{\preccurlyeq}

\newcommand{\HP}{\mathcal{HP}}
\newcommand{\BA}{\bold{A}}

\newcommand{\an}{\text{\, and \,}}

\newcommand{\ch}{\mathbbm{1}}

\begin{document}

\title[Hyper-positive definite functions]{Hyper-positive definite functions I: scalar case,  \\ Branching-type stationary stochastic processes}

\author
{Yanqi Qiu}
\address
{Yanqi QIU: Institute of Mathematics and Hua Loo-Keng Key Laboratory of Mathematics, AMSS, Chinese Academy of Sciences, Beijing 100190, China.}
\email{yanqi.qiu@amss.ac.cn; yanqi.qiu@hotmail.com}

\author{Zipeng Wang}
\address{Zipeng Wang: Department of Mathematics, Soochow University, Suzhou 215006, China.}
\email{zipengwang2012@gmail.com}

\begin{abstract}
We propose a definition of branching-type stationary stochastic processes on rooted trees and  related  definitions of hyper-positivity for functions on the unit circle and functions on the set of non-negative integers. We then obtain (1) a necessary and sufficient condition on a rooted tree for the existence of non-trivial branching-type stationary stochastic processes on it, (2) a complete criterion of the hyper-positive functions in the setting of rooted homogeneous trees in terms of  a variant of the classical Herglotz-Bochner Theorem, (3) a prediction theory result for branching-type stationary stochastic processes.

 As an unexpected application,   we obtain  natural hypercontractive inequalities for Hankel operators with hyper-positive symbols.
\end{abstract}


\maketitle

\setcounter{equation}{0}

\section{Introduction}

\subsection{Branching-type stationary stochastic processes on rooted trees}

In this paper, we propose a definition of {\it branching-type stationary stochastic processes (abbr. branching-type SSP) on rooted trees}. One of our main results is the full classification of such new processes on {\it rooted homogeneous trees}.

\begin{center}
\begin{tikzpicture}

\filldraw (0,10) circle (2pt);
\node (e) at (0-0.1,10+0.3) {$e = \text{root vertex}$} ;
\draw (0,10) --(-2,9);
\filldraw (-2, 9) circle (2pt);
\draw (-2,9) --(-3,8);
\draw (-2,9) --(-1,8);
\filldraw (-3, 8) circle (2pt);
\node (e) at (-3-0.1,8+0.3) {$\tau$} ;
\draw (-3,8) --(-3.5,7);
\draw (-3,8) --(-2.5,7);
\filldraw (-3.5, 7) circle (2pt);
\draw (-3.5,7) --(-3.8,6);
\draw (-3.5,7) --(-3.2,6);
\filldraw (-3.8, 6) circle (2pt);
\draw[dashed](-3.8,6)--(-3.6,5);
\draw[dashed](-3.8,6)--(-4,5);
\filldraw (-3.2, 6) circle (2pt);
\draw[dashed](-3.2,6)--(-3,5);
\draw[dashed](-3.2,6)--(-3.4,5);
\filldraw (-2.5, 7) circle (2pt);
\draw(-2.5,7)--(-2.8,6);
\draw(-2.5,7)--(-2.2,6);
\filldraw (-2.8, 6) circle (2pt);
\draw[dashed](-2.8,6)--(-2.6,5);
\draw[dashed](-2.8,6)--(-3,5);
\filldraw (-2.2, 6) circle (2pt);
\draw[dashed](-2.2,6)--(-2,5);
\draw[dashed](-2.2,6)--(-2.4,5);
\filldraw (-1, 8) circle (2pt);
\draw(-1,8)--(-1.5,7);
\draw(-1,8)--(-0.5,7);
\filldraw (-1.5, 7) circle (2pt);
\draw(-1.5,7)--(-1.8,6);
\draw(-1.5,7)--(-1.2,6);
\filldraw (-1.8, 6) circle (2pt);
\draw[dashed](-1.8,6)--(-1.6,5);
\draw[dashed](-1.8,6)--(-2,5);
\filldraw (-1.2, 6) circle (2pt);
\draw[dashed](-1.2,6)--(-1,5);
\draw[dashed](-1.2,6)--(-1.4,5);
\filldraw (-0.5, 7) circle (2pt);
\draw(-0.5,7)--(-0.8,6);
\draw(-0.5,7)--(-0.2,6);
\filldraw (-0.8, 6) circle (2pt);
\draw[dashed](-0.8,6)--(-0.6,5);
\draw[dashed](-0.8,6)--(-1,5);
\filldraw (-0.2, 6) circle (2pt);
\draw[dashed](-0.2,6)--(-0.4,5);
\draw[dashed](-0.2,6)--(-0,5);
\filldraw (2, 9) circle (2pt);
\draw[color=red!60, fill=red!5, thick]  (0,10) --(2,9);
\draw[color=blue!60, fill=blue!5, thick]  (0.08,10) --(2.08,9);
\node (e) at (2-0.1,9+0.3) {$\sigma$} ;
\draw[color=red!60, fill=red!5, thick]  (2,9) --(1,8);
\draw[color=blue!60, fill=blue!5, thick]  (2.05,9) --(1.05,8);
\filldraw (1, 8) circle (2pt);
\draw[color=red!60, fill=red!5, thick]  (1,8) --(1.5,7);
\draw[color=blue!60, fill=blue!5, thick]  (1,8) --(0.5,7);
\filldraw (1.5, 7) circle (2pt);
\node (e) at (1.5+0.1,7+0.3) {$\delta$} ;
\draw[color=red!60, fill=red!5, thick]  (1.5,7) --(1.8,6);
\draw(1.5,7)--(1.2,6);
\filldraw (1.8, 6) circle (2pt);
\draw[color=red!60, fill=red!5, thick,dashed](1.8,6)--(1.6,5);
\draw[dashed](1.8,6)--(2,5);
\filldraw (1.2, 6) circle (2pt);
\draw[dashed](1.2,6)--(1.4,5);
\draw[dashed](1.2,6)--(1,5);
\filldraw (0.5, 7) circle (2pt);
\draw[color=blue!60, fill=blue!5, thick]  (0.5,7) --(0.2,6);
\draw(0.5,7)--(0.8,6);
\filldraw (0.8, 6) circle (2pt);
\draw[dashed](0.8,6)--(0.6,5);
\draw[dashed](0.8,6)--(1.0,5);
\filldraw (0.2, 6) circle (2pt);
\draw[color=blue!60, fill=blue!5, thick,dashed](0.2,6)--(0.4,5);
\draw[dashed](0.2,6)--(0,5);
\draw(2,9)--(3,8);
\filldraw (3, 8) circle (2pt);
\draw(3,8)--(2.5,7);
\draw(3,8)--(3.5,7);
\filldraw (2.5, 7) circle (2pt);
\draw(2.5,7)--(2.8,6);
\draw(2.5,7)--(2.2,6);
\filldraw (2.8, 6) circle (2pt);
\draw[dashed](2.8,6)--(3,5);
\draw[dashed](2.8,6)--(2.6,5);
\filldraw (2.2, 6) circle (2pt);
\draw[dashed](2.2,6)--(2,5);
\draw[dashed](2.2,6)--(2.4,5);
\filldraw (3.5, 7) circle (2pt);
\draw(3.5,7)--(3.8,6);
\draw(3.5,7)--(3.2,6);
\filldraw (3.8, 6) circle (2pt);
\draw[dashed] (3.8,6) -- (4.05,5);
\draw[dashed] (3.8,6) -- (3.6,5);
\filldraw (3.2, 6) circle (2pt);
\draw[dashed] (3.2,6) -- (3.4,5);
\draw[dashed] (3.2,6) -- (3.05,5);
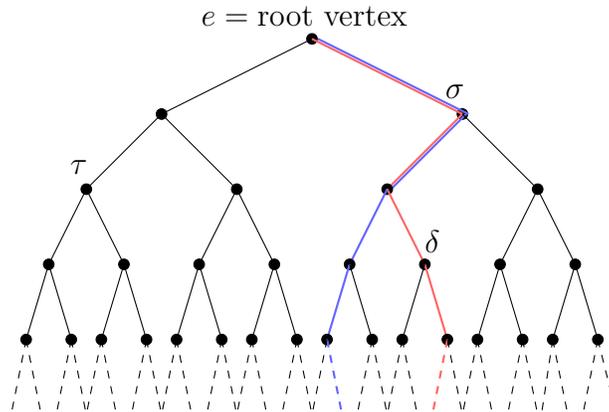
\node at (0,4) {Figure 1: Examples of rooted geodesic rays in $T_2$ (blue ray and red ray)};
\end{tikzpicture}
\end{center}

Let $T$ be the set of vertices of an {\it infinite rooted tree} with a distinguished vertex $e\in T$ (called the root of the tree).  $T$ is  equipped with the usual graph distance $d(\cdot, \cdot)$. In what follows, we always make the following assumption.
{\flushleft \bf Assumption on $T$}: the rooted tree $T$ has no leaves, that is, any vertex of $T$ has at least one descendant.

 There is a natural partial order  $\pless$ on  $T$ making the root the smallest element: two vertices of $T$ are comparable if and only if they are in the same rooted geodesic ray and we say $\sigma_1 \pless \sigma_2$ if two vertices $\sigma_1, \sigma_2\in T$ are in the same rooted geodesic ray with $d(\sigma_1, e) \le d(\sigma_2, e)$.   For instance, in Figure 1, the vertices $\sigma, \delta$ are comparable with $\sigma \pless \delta$ while the vertices $\tau, \delta$ are non-comparable.  

Before giving the definition of branching-type stationary stochastic processes on $T$, we first briefly recall the classical weak stationary stochastic processes  (abbr. SSP)  on the set $\N = \{0, 1, 2, \cdots\}$ of non-negative integers : a {\it mean-zero} stochastic process $X = (X_n)_{n\in \N}$ on  $\N$ is called weak stationary if all the random variables $X_n$ admit finite second moments and the covariance $\Cov(X_n, X_{n+k})  = \E(X_n \overline{X_{n+k}})$,  for any $n, k\in \N$, is independent of $n$. By the well-known Herglotz-Bochner Theorem,  any SSP $X= (X_n)_{n\in \N}$  on $\N$ admits a spectral measure, denoted by $\nu_X$,  on the unit circle $\T= \R/2 \pi \Z$ which is the unique positive Radon measure on $\T$ characterized by the formula 
\begin{align}\label{def-spec}
\Cov(X_n, X_{n+k}) = \widehat{\nu_X}(k) = \int_\T  e^{-i  k \theta} d \nu_X(\theta), \quad  \forall n, k \in \N.
\end{align}

Set 
\[
\partial T : = \Big\{\xi\Big|\text{$\xi$ is a geodesic ray in $T$ starting from the root}\Big\}.
\]
Note that, each rooted geodesic ray $\xi \in \partial T$ can be canonically identified with the set $\N$.

\begin{definition}[Branching-type SSP]
A mean-zero square-integrable stochastic process $(X_\sigma)_{\sigma \in T}$ on $T$ is called a branching-type SSP if 
\begin{itemize}
\item Restricted on every rooted geodesic ray $\xi \in \partial T$, we have a classical stationary  stochastic process $X^\xi: = (X_\sigma)_{\sigma \in \xi}$, by identifying canonically the subset $\xi\subset T$ with the set $\N$.
\item The family of these stationary stochastic processes $X^\xi, \xi \in \partial T$ share a common spectral measure. That is, 
\begin{align}\label{spec-BSSP}
\nu_{X^\xi} = \nu_{X^{\xi'}}, \quad \forall \xi, \xi' \in \partial T.
\end{align}
\item For any pair of non-comparable vertices $\sigma, \tau \in T$, the random variables $X_\sigma, X_\tau$ are un-correlated. That is, 
\begin{align}\label{def-uncor}
\Cov(X_\sigma, X_\tau) = 0, \quad \text{for all non-comparable $\sigma, \tau \in T$.}
\end{align}
\end{itemize}
\end{definition}

We define the spectral measure $\mu_X$ of a branching-type SSP $X = (X_\sigma)_{\sigma \in T}$ the common spectral measure in  \eqref{spec-BSSP}. That is, 
\begin{align}\label{def-mu-X}
\mu_X:  = \nu_{X^\xi}, \quad \forall \xi \in \partial T.
\end{align}

Clearly, given a rooted tree $T$, all the spectral information is carried on its spectral measure.  Note also that any branching-type SSP $(X_\sigma)_{\sigma \in T}$ on $T$ has a Gaussian version $(Y_\sigma)_{\sigma \in T}$, namely, $(Y_\sigma)_{\sigma \in T}$ is a centered Gaussian process on $T$ such that 
\[
\Cov(Y_\sigma, Y_\tau) = \Cov(X_\sigma, X_\tau), \quad \forall \sigma, \tau \in T.
\]
In particular,  in the case of Gaussian branching-type SSP, the condition \eqref{def-uncor} for $(Y_\sigma)_{\sigma \in T}$ becomes the condition that $Y_\sigma$ and $Y_\tau$ are independent for any pair of non-comparable $\sigma, \tau \in T$.

A trivial example of branching-type SSP on $T$ is the i.i.d. centered Gaussian process on $T$. We say that a branching-type SSP on $T$ is {\it non-trivial} if there exists a pair $(\sigma, \delta)$ of vertices of $T$ such that $\Cov(X_\sigma, X_\delta) \ne 0$.

\begin{theorem}[General rooted trees]\label{thm-general-T}
There exists a non-trivial branching-type SSP on $T$  if and only if $T$ is of uniformly bounded valence (i.e., there is a uniform bound on the number of edges adjacent to all vertices of $T$).
\end{theorem}

The existence of non-trivial branching-type SSP on a rooted tree with uniformly bounded valence will be proved by embedding this tree into a rooted homogeneous one and we then use the following complete criterion (Theorem \ref{thm-hom-T}) for the spectral measure of a branching-type SSP on a  rooted homogeneous tree.

For any $q \ge 2$, let $T_q$ denote the set of vertices of the rooted $q$-homogeneous tree (see Figure 1 for the picture of the rooted $2$-homogeneous tree). We have the following necessary and sufficient condition for a measure on $\T$ to be the spectral measure of a branching-type SSP on $T_q$.

Recall that for any $r\in (0, 1)$,  the Poisson convolution $P_{r} * \mu$ of a positive Radon measure $\mu$ is defined by  
\[
(P_{r} * \mu)(e^{i\theta}): = \int_\T P_{r} (e^{i (\theta - \theta')}) d\mu(\theta'),
\]
where  $P_{r}: \T\rightarrow \R_{+}$ is the Poisson kernel at the point $r \in \D$ given by 
\begin{align}\label{def_Poisson}
P_{r} (e^{i \theta}) = \frac{1}{2 \pi} \sum_{n\in \Z}  r^{|n|} e^{i n \theta}=  \frac{1}{2 \pi} \frac{1 - r^2}{| 1 - r e^{i  \theta}|^2}.
\end{align}

\begin{theorem}[Rooted homogeneous trees]\label{thm-hom-T}
Let $q\ge 2$ be an integer.  A measure $\nu$ on $\T$ is the spectral measure of a branching-type SSP on $T_q$ if and only if there exists a positive Radon measure $\mu$ on $\T$ such that 
\[
\nu = P_{1/\sqrt{q}}* \mu. 
\] 
\end{theorem}

Theorem \ref{thm-hom-T} has the following immediate corollary. 

\begin{corollary}\label{cor-spec-meas}
Let $q\ge 2$ be an integer.
The spectral measure of any branching-type SSP  on $T_q$ is absolutely continuous with respect to the Lebesgue measure on $\T$ and the corresponding Radon-Nikodym derivative is real-analytic.
\end{corollary}

Corollary \ref{cor-spec-meas} has the following extension. Given a rooted tree $T$ and any integer $n\ge 1$, define 
\[
\Delta_n(T): = \sup_{\sigma \in T} \# \Big\{  \delta \in T \Big| \text{$\sigma \pless \delta$ and $d(\sigma, \delta) = n$} \Big\}.
\]
In other words, $\Delta_n(T)$ is the maximal number of $n$-th descendants for a vertex in $T$. 

\begin{proposition}\label{prop-ac}
Suppose that 
\begin{align}\label{inv-sum}
\sum_{n= 1}^\infty \frac{1}{\Delta_n(T)}<\infty.
\end{align}
Then the spectral measure of any branching-type SSP on $T$ is absolutely continuous with respect to the Lebesgue measure $m$ on $\T$ and the corresponding Radon-Nikodym derivative is in $L^2(\T, m)$. 
\end{proposition}

\subsection{Branching-type SSP as extension of one-dimensional SSP}\label{sec-extension}
One can look at the branching-type SSP from another point of view which we  now describe. 

Given two index sets $I, J$ with $I \subset J$.  We say a stochastic process $(X_i)_{i\in I}$ is the restriction of (or can be embedded into) another stochastic process $(Y_j)_{j\in J}$ if 
\[
X_i = Y_i, \quad \forall i \in I.
\]
If a stochastic process $(X_i)_{i\in I}$ can be embedded into another stochastic process $(Y_j)_{j\in J}$,  we will also say that $(Y_j)_{j\in J}$ is an extension of $(X_i)_{i\in I}$. 

The study of branching-type SSP is then related to the problem of embedding a classical SSP into a stochastic process with a particular structure on a tree. More precisely, we can formulate the following problem.

\begin{problem}[Embedding of SSP]\label{pb-SSP}
 Consider a classical weak stationary stochastic process $(X_n)_{n\in \N}$ on $\N$ and a rooted tree $T$ without leaves. Choose any rooted geodesic ray $\xi \in \partial T$. By writing $\xi  = (\sigma_n)_{n\in \N}$, we define a stochastic process $X^\xi$ indexed by the vertices in $\xi$ by 
\[
X^\xi_{\sigma_n}: = X_n, \quad \forall n \in \N.
\]
Under which condition  on the spectral measure of $(X_n)_{n\in \N}$ and the rooted tree $T$ can the process $X^\xi$ be embedded into a branching-type SSP on $T$ ? 
\end{problem}

For the rooted homogeneous tree, Theorem \ref{thm-hom-T} gives the complete solution of Problem \ref{pb-SSP}. For other rooted tree, Problem \ref{pb-SSP} could be very complicated. Theorem \ref{thm-2-1} below shows that even for very simple trees, Problem \ref{pb-SSP} may have a non-trivial solution.

For any positive integer $q\ge 2$, let $T(q; 1)$ denote the  rooted tree such that the root vertex has exactly $q$-descendants and all the other vertices have exactly $1$-descendant. 

\begin{center}
\begin{tikzpicture}
\filldraw(-12,10) circle (2pt);
\node (e) at (-12-0.2,10+0.1) {$e$} ;
\filldraw(-11.65,9) circle (2pt);
\filldraw(-12.35,9) circle (2pt);
\draw (-12,10)--(-11.65,9);
\draw (-12,10)--(-12.35,9);
\filldraw(-12.65,8) circle (2pt);
\draw (-11.65,9)--(-11.35,8);
\draw (-12.35,9)--(-12.65,8);
\filldraw(-11.35,8) circle (2pt);
\draw (-11.35,8)--(-11.2,7.5);
\draw(-12.65,8)--(-12.8,7.5);
\draw[dashed](-11.2,7.5)--(-11,6.8);
\draw[dashed](-12.8,7.5)--(-13,6.8);
\node at (-12,6) {Figure 2: $T(2; 1)$};
\end{tikzpicture}
\end{center}
\begin{remark}
Note that although we can identify naturally the set $T(2; 1)$ with the set $\Z$,  the partial order on $T(2; 1)$ is different from the usual order on $\Z$.   
\end{remark}

\begin{theorem}\label{thm-2-1}
For $q \ge 2$, a positive Radon measure $\mu$ is the spectral measure  of a branching-type SSP on $T(q; 1)$ if and only if 
\begin{align}\label{T-q-1}
\exp \left(\int_\T  \log  \Big( \frac{d\mu_{ac}}{dm}\Big)  dm\right) \ge \Big(1 - \frac{1}{q}\Big)\mu(\T),
\end{align}
where $\mu_{ac}$ is the absolutely continuous part of $\mu$ with respect to the normalized Haar measure $dm$ on $\T$.  
\end{theorem}

The condition \eqref{T-q-1} implies in particular that 
$\frac{d\mu_{ac}}{dm}(\theta) > 0, a.e.$
and thus any spectral measure of a branching-type SSP on $T(q; 1)$ must have full support on $\T$. For the absolutely continuous positive Radon measures on $\T$ with density function $\varphi \ge 0$, the condition \eqref{T-q-1} reads as 
\begin{align}\label{T-q-1-ab}
\exp \left(\int_\T  \log   \varphi \right) \ge \Big(1 - \frac{1}{q}\Big)\int_\T\varphi,
\end{align}
which can be viewed as an inverse Jensen-inequality with a multiplicative constant (here and after, the notation for the integral $\int_\T \varphi dm$ is simplified as $\int_\T \varphi$). By an elementary computation, if $\varphi = a \mathds{1}_{A} + b \mathds{1}_{\T\setminus A}$ with $m(A)= 1/2$ and $a> 0, b> 0$, then  $\varphi$ satisfies the condition \eqref{T-q-1-ab} if and only if 
\[
 \frac{q-1}{q + \sqrt{2q-1}} \le \frac{a}{b} \le  \frac{q + \sqrt{2q-1}}{q-1}.
\]
One can easily check that if $g: \T\rightarrow \R$ satisfies 
\[
\| g\|_\infty \le \frac{1}{2}\log\Big(\frac{q}{q-1}\Big),
\]
then the function $\varphi = e^g$ satisfies the condition \eqref{T-q-1-ab}. 

\begin{remark}
Theorem \ref{thm-2-1} implies that if an integrable function  $\varphi\ge  0$ satisfies the  strict inequality
\[
\exp \left(\int_\T  \log  \varphi \right) >  \Big(1 - \frac{1}{q}\Big)\int_\T \varphi,
\]
then for any singular  positive Radon measure $\mu_s$ on $\T$ such that 
\[
\mu_s(\T) \le \frac{\exp \left(\int_\T  \log  \varphi \right) }{1- 1/q} - \int_\T \varphi,
\]
the measure 
$
\mu: = \varphi dm + \mu_s
$ 
is the spectral measure of a branching-type SSP on $T(q; 1)$. That is, in contrary to the case of the rooted $q$-homogeneous trees $T_q$,  the spectral measure of a branching-type SSP on $T(q; 1)$ could have a non-trivial singular part. 
\end{remark}

\begin{remark}
For any integer $q \ge 2$,   since $T(q;1)$ is a sub-rooted-tree of  the rooted $q$-homogeneous tree $T_q$, Theorem \ref{thm-hom-T} and Theorem  \ref{thm-2-1} implies in particular that  for any probability measure $\mu$ on $\T$, we have
\[
\int_\T \log  (P_{1/\sqrt{q}} * \mu) \ge  \log (1 - 1/q). 
\]
The above inequality can also be proved directly (we omit the details) and has the following generalization: for any probability measure $\mu$ on $\T$ and any $r\in (0, 1)$, 
\[
\int_\T \log  (P_{r} * \mu) \ge  \log (1 - r^2). 
\]
\end{remark}

\bigskip

For strong stationary stochastic processes (abbr. SSSP), one can study similar embedding problem. Therefore, we give the following definition.

\begin{definition}[Branching-type SSSP]
A stochastic process $(X_\sigma)_{\sigma \in T}$ on $T$ is called a branching-type SSSP if 
\begin{itemize}
\item Restricted on every rooted geodesic ray $\xi \in \partial T$, we have a classical strong stationary  stochastic process $X^\xi: = (X_\sigma)_{\sigma \in \xi}$, by identifying canonically the subset $\xi\subset T$ with the set $\N$.
\item The family of these strong stationary stochastic processes $X^\xi, \xi \in \partial T$ share a common distribution. That is, for any pair $(\xi, \xi')$ of rooted geodesic rays, by using the natural identifications $\xi \simeq \N \simeq \xi'$, we have 
\[
(X_\sigma)_{\sigma \in \xi} \stackrel{d}{=}  (X_{\sigma'})_{\sigma'\in \xi'}.
\]
\item For any pair of non-comparable vertices $\sigma, \tau \in T$, the random variables $X_\sigma, X_\tau$ are independent. 
\end{itemize}
\end{definition}

\begin{remark}
Note that in the definition of branching-type SSSP, we do not require the integrability of the processes. Note also that the joint distribution of a branching-type SSSP in general can not be determined by the common distribution of this process restricted on a rooted geodesic ray. 
\end{remark}

\begin{problem}[Embedding of SSSP]\label{pb-SSSP}
 Consider a classical strong stationary stochastic process $(X_n)_{n\in \N}$ on $\N$ and a rooted tree $T$. Choose any rooted geodesic ray $\xi \in \partial T$. By writing $\xi  = (\sigma_n)_{n\in \N}$, we define a stochastic process $X^\xi$ indexed by $\xi$ by 
\[
X^\xi_{\sigma_n}: = X_n, \quad \forall n \in \N.
\]
Under which condition  on the distribution of $(X_n)_{n\in \N}$ and the rooted tree $T$ can the process $X^\xi$ be embedded into a branching-type SSSP on $T$ ? 
\end{problem}

Clearly, for centered Gaussian SSP, Problem \ref{pb-SSSP} is reduced to Problem \ref{pb-SSP}.  Assume that $(X_\sigma)_{\sigma\in T}$ is a branching-type SSSP on $T$ which is a Gaussian process. Then for any measurable function $F: \C\rightarrow \C$, the stochastic process $(F(X_\sigma))_{\sigma \in T}$ is a branching-type SSSP. Except for this simple construction from Gaussian processes, in general,  it is not obvious how to construct a natural non-trivial branching-type SSSP. In a forthcoming paper, we will construct a natural family of determinantal point processes on certain rooted trees which are braching-type SSSP.

\subsection{Hyper-positive definite functions: outline of the proof of Theorem \ref{thm-hom-T}}

It is convenient for us to identity $T_q$ with the unital free semi-group $\F_q^{+}$ on generators $s_1,\cdots, s_q$. The neutral element of $\F_q^{+}$ will be denoted $e$. For any element $\sigma = s_{i_1}s_{i_2} \cdots s_{i_n} \in \F_q^{+}$, we set $|\sigma| = n$  and by convention, we set $|e| =0$.  The natural partial order on $T_q$ is then the same as the partial order $\preceq$ on $\F_q^{+}$ described as follows: for any $\sigma, \delta \in \F_q^{+}$, 
\[
\sigma \preceq \delta \Longleftrightarrow \text{$\delta = \sigma w$ for some $w\in \F_q^{+}$}.
\]

Our proof of Theorem \ref{thm-hom-T} relies on a variant of Herglotz-Bochner Theorem. First recall the classical Herglotz-Bochner Theorem.   For any set $\Sigma$, we say that a kernel $K: \Sigma \times \Sigma \rightarrow \C$ is positive definite if for each $k\in \N$, each choice of elements $\sigma_1, \cdots, \sigma_k \in \Sigma$, the matrix $[ K(\sigma_i, \sigma_j)]_{1\le i, j \le k}$ is non-negative definite.  A function $\alpha: \N \rightarrow \C$ is called positive definite, if  the associated {\it Toeplitz kernel} $\TT_\alpha^{(\N)}: \N\times \N \rightarrow \C$ is positive definite, where $\TT_\alpha^{(\N)}$ is the kernel defined by 
\begin{align}\label{def-K-N}
\TT_\alpha^{(\N)}(k, \ell) : = \alpha(k-\ell), \quad k, \ell \in \N, 
\end{align}
where we use the following convention:
\[
\alpha(-n):=\overline{\alpha(n)}, \quad \text{for $n\ge 1$}. 
\]
The classical Herglotz-Bochner Theorem (cf.  Bochner \cite[Theorem 3.2.3]{Bochner1955}) says that any positive definite  function on $\N$ is the restriction of the Fourier series on $\N$ of a unique positive Radon measure on the circle.

 Note that a function $\alpha:\N\rightarrow \C$ is positive definite if and only if there exists a SSP $X= (X_n)_{n\in \N}$ on $\N$ with spectral measure $\nu_X$ satisfying 
\[
\alpha(k)  = \widehat{\nu_X} (k), \quad \forall k \in \N.
\]

\begin{definition}\label{def-HPD}
Let $q\ge 2$ be an integer. A function $\alpha: \N\rightarrow \C$ is called $q$-hyper-positive definite (abbr. $q$-HPD) if there exists a branching-type SSP $X = (X_\sigma)_{\sigma\in T_q}$ on $T_q$ with spectral measure $\mu_X$ defined in \eqref{def-mu-X} such that 
\[
\alpha(k) = \widehat{\mu_X}(k), \quad k \in \N.
\]
\end{definition}

We shall use the equivalent definition of $q$-HPD functions in Lemma \ref{lem-HPD} below.


Given any function $\alpha: \N \rightarrow \C$, we define a  {\it branching-Toeplitz kernel} $\TT_\alpha^{(\F_q^{+})}$ on $\F_q^{+}$ by
\begin{align}\label{def_kernel_a}
\TT_\alpha^{(\F_q^{+})}(\sigma, \delta)  := \left\{
\begin{array}{ll}
\alpha(|w|); & \text{if $\delta = \sigma w$ for some $w \in \F_q^{+}$} \vspace{2mm}
\\
\overline{\alpha(|w|)}; & \text{if $\sigma = \delta w$ for some $w \in \F_q^{+}$} \vspace{2mm}
\\
0; & \text{otherwise}
\end{array}
\right.
\end{align}

\begin{remark} 
Our definition \eqref{def_kernel_a} of Toeplitz type kernel is slightly different from Popescu \cite{Popescu-96} in that we do not require the condition $K(e, e) = 1$ and  different from Popescu \cite{Popescu-99} in that we add the additional requirement $K(\sigma, \delta) = 0$ for non-comparable pair  $(\sigma, \delta)$. 
\end{remark}

Clearly, by the definition of branching-type SSP, we have

\begin{lemma}\label{lem-HPD}
Let $q\ge 2$ be an integer.  A function $\alpha: \N \rightarrow \C$ is $q$-HPD if and only if  the branching-Toeplitz  kernel $\TT_\alpha^{(\F_q^{+})}$ on $\F_q^{+}$ defined in \eqref{def_kernel_a} is positive definite, or equivalently, there exists a branching-type SSP $(X_\sigma)_{\sigma\in \F_q^{+}}$ on $\F_q^{+}$  with
\begin{align}\label{GF_K}
\Cov(X_\sigma, X_\delta) = \E(X_\sigma \overline{X_\delta}) = \TT_\alpha^{(\F_q^{+})}(\sigma, \delta),\quad \sigma, \delta \in \F_q^{+}.  
\end{align}
\end{lemma}

In general,  it is difficult to determine whether a given branching-Toeplitz kernel on $\F_q^{+}$ is positive definite or not. Even the existence of non-identically zero $q$-HPD functions is not immediately obvious. The key example of non-trivial $q$-HPD function $\beta_q:\N\rightarrow \R$ is provided in the following
\begin{lemma}\label{prop_Bozej}
Let $q\ge 2$ be an integer and define 
\begin{align}\label{def_beta_q}
\beta_q(n) = q^{-n/2},\quad n \in \N.
\end{align}
Then the function $\beta_q$  is $q$-HPD.  
\end{lemma}

In \S \ref{sec_strongpdf} below, we will present three different proofs of Lemma \ref{prop_Bozej}, all of which have their own interests: 1) the first proof is based on a beautiful result of Bo\.{z}ejko \cite{Bozejko-89} on Markov product of positive definite kernels on union of two sets (see also \cite{Q-extension} for an alternative proof of this result); 2) the second probabilistic proof is based on explicit constructions of Gaussian processes indexed by $\F_q^{+}$ and on the self-similar structure of the Cayley graph of $\F_q^{+}$; 3) the third proof is direct and uses the natural Cantor measure on the boundary $\partial T_q$ of the rooted tree $T_q$.

All $q$-HPD functions are given by the following

\begin{theorem}[Herglotz-Bochner Theorem for $q$-HPD functions]\label{thm_classification}
A function $\alpha: \N \rightarrow \C$ is $q$-HPD if and only if the function 
\[
\N\ni n \mapsto q^{n/2} \alpha(n)\in \C
\]
 is positive definite on $\N$. Therefore, $\alpha$ is $q$-HPD if and only if there exists a positive Radon measure $\nu$ on  $\T$ such that 
\[
\alpha(n ) = q^{-n/2} \widehat{\nu}(n) =  q^{-n/2} \cdot\int_\T  e^{-i  n \theta} d \nu(\theta), \quad n \in \N.
\]
\end{theorem}

Theorem \ref{thm_classification} implies that any $q$-HPD function  $\alpha: \N\rightarrow \C$ has an exponential decay:
\begin{align}\label{exp_decay}
|\alpha(n)| \le \alpha(0) q^{-n/2}, \quad n \in \N.
\end{align}

In Theorem \ref{thm_classification},  the positive measure $\nu$ is uniquely determined by the $q$-HPD function $\alpha$. In what follows, for emphasizing the dependence of $\alpha$,  we will  denote by $\nu_\alpha$  the measure determined by a $q$-HPD function. Thus we have 
\begin{align}\label{def_mu_a}
\widehat{\nu}_\alpha (n)  = q^{n/2} \alpha(n), \quad n \in \N. 
\end{align}

\subsection{Hyper-positive functions on the circle}

\begin{definition}\label{def-hp-circle}
An integrable function $\varphi: \T \rightarrow \C$ is called $q$-hyper-positive (abbr. $q$-HP), if the function $\N \ni n \mapsto \widehat{\varphi}(n)$ is  $q$-HPD. 
\end{definition}

Theorem  \ref{thm_classification} now implies the following criterion of $q$-HP functions.
\begin{theorem}[Criterion of $q$-HP functions]\label{thm_class_bis}
A function $\varphi: \T\rightarrow \C$ is $q$-HP if and only if there exists a positive Radon measure $\mu$ on $\T$ such that $\varphi = P_{1/\sqrt{q}} * \mu$. Moreover, the positive measure $\mu$ is uniquely determined by $\varphi$.  
\end{theorem}

Theorem \ref{thm_class_bis} implies in particular that any $q$-HP function is {\it real-analytic} and is associated with a unique $q$-HPD function $\alpha$ on $\N$ via the formula: 
\begin{align}\label{def_phi}
\varphi_\alpha(\theta) = \sum_{n\in \Z}  \alpha(n) e^{i   n \theta},\, \text{where $\alpha(-n) : = \overline{\alpha(n)}$  for any $n \in \N$.}
\end{align}

\subsection{A prediction theory of branching-type SSP}

For any branching-type SSP $X = (X_\sigma)_{\sigma \in \F_q^{+}}$ on $\F_q^{+}$ (or equivalently on  $T_q$), using the classical Szeg\"o First Theorem, we will compute
\[
d_{L^2} \left(X_e, \,\overline{\spann}^{L^2}\big\{X_\sigma:  \sigma\in \F_q^{+}\setminus \{e\}\big\} \right),
\]
 the  $L^2$-distance from $X_e$ to the closed linear span by all  random variables $X_\sigma$ with $\sigma \in \F_q^{+}\setminus \{e\}$  in the $L^2$ space of the underlying probability space. 

\begin{theorem}\label{thm-pred}
Let $(X_\sigma)_{\sigma\in \F_q^{+}}$ be a branching-type SSP  on $\F_q^{+}$ with spectral measure $\mu_X$.  Let $\alpha: \N \rightarrow \C$ be the corresponding $q$-HPD function on $\N$  such that the equality \eqref{GF_K} is satisfied. Let $\nu_\alpha$ be the unique positive Radon measure on $\T$ determined  by \eqref{def_mu_a}  and write the Lebesgue decomposition 
\begin{align}\label{leb-dec}
d\nu_\alpha(\theta) = w_\alpha(\theta) d m(\theta) + d\nu_\alpha^{s}(\theta),
\end{align}
where $m$ is the normalized Haar measure on $\T$.   Then 
\begin{align}\label{pred-all}
d_{L^2} \left(X_e, \, \overline{\spann}^{L^2}\big\{X_\sigma:  \sigma\in \F_q^{+}\setminus \{e\}\big\} \right)   = \exp \left(\frac{1}{2}\int_\T  \log  w_\alpha(\theta)  dm(\theta)\right).
\end{align}
\end{theorem}

Recall the definition of $T(q; 1)$ in  \S \ref{sec-extension} and Theorem \ref{thm-2-1} for the classification of spectral measure of branching-type SSP on $T(q;1)$. For branching-type SSP on $T(q; 1)$, we have the following
\begin{proposition}\label{prop-pred-T-q-1}
Let $q \ge 2$ be an integer and  let $(X_\sigma)_{\sigma \in T(q; 1)}$ be a branching-type SSP on $T(q; 1)$ with spectral measure $\mu$. Then 
\begin{align}\label{pred-T-q-1}
d_{L^2} \left(X_e, \, \overline{\spann}^{L^2}\big\{X_\sigma:  \sigma\in T(q; 1) \setminus \{e\}\big\} \right)   = \sqrt{ q \exp \left(\int_\T  \log  \Big( \frac{d\mu_{ac}}{dm}\Big) dm\right)  - (q-1) \mu(\T)}.
\end{align}
\end{proposition}

\subsection{Application in hyper-contractive Hankel inequalities}

From the definition of HP functions, any $q$-HP function is naturally associated with a centered stochastic process on $\F_q^{+}$ satisfying \eqref{GF_K}.  By restricting this stochastic process on some sub-trees of $\F_q^{+}$, we can construct classical SSP's on $\N$ and obtain the following {\it hyper-contractive} inequalities of {\it Hankel operators}. These inequalities seem to be new in the litterature.  Other similar inequalities can also be found using our procedure.  The details are explained in \S \ref{sec_hankel}.   The analogue of these inequalities for matrix-valued functions will be treated in a forthcoming paper. 
The Hankel operator is recalled as follows. The Hardy space $H^2(\T)$ is given  by 
\[
H^2(\T): = \Big\{ f (e ^{i\theta}) = \sum_{n =0}^\infty a_n e^{i n \theta}\Big| a_n\in \C \an \sum_{n =0}^\infty |a_n|^2 <\infty\Big\}.
\]
Let $H^2_0(\T)\subset H^2(\T)$ denote the subspace of functions $f\in H^2(\T)$ with $\widehat{f}(0) = 0$ and define  $H^2_{-}(\T): =  L^2(\T) \ominus H_0^2(\T)$, that is, 
\[
H_{-}^2(\T): = \Big\{ f (e ^{i\theta}) = \sum_{n =0}^\infty a_n e^{-i n \theta}\Big| a_n\in \C \an \sum_{n=0}^\infty |a_n|^2 <\infty\Big\}.
\] 
Let $R_{-}: L^2(\T) \rightarrow H_{-}^2(\T)$ be the orthogonal projection onto $H_{-}^2(\T)$ (called Riesz projection). Given a function $\varphi \in L^2(\T)$, we define the Hankel operator $H_\varphi: H^2_0(\T) \rightarrow H^2_{-}(\T)$ on dense subset of analytic trigonometric polynomials  by 
\begin{align}\label{def_hankel}
H_\varphi(f)= R_{-} (\varphi f).
\end{align}
More explicitly, for an analytic trigonometric polynomial $f = \sum_{n\ge 0} a_n e^{i n \theta}$, we have 
\begin{align}\label{def_h_f}
H_\varphi (f)  =  \sum_{k = -\infty}^0 \Big(  \sum_{n =-\infty}^{k} \widehat{\varphi}(n) a_{k-n} \Big) e^{i k \theta}
\end{align}
 Note that our definition of Hankel operators is slightly different from the classical way of definition, see Peller \cite[Chapter 1]{Peller-book}.  The Hankel operators in \cite{Peller-book} are defined as an operators from $H^2(\T)$ to $L^2(\T)\ominus H^2(\T)$ and are, using the notation of this paper, given by $f \mapsto e^{-i\theta} H_\varphi ( e^{i\theta}f)$.

\begin{proposition}\label{prop_hankel_1}
Assume that  $\varphi: \T\rightarrow \R_{+}$ is $2$-HP and is not identically zero. Then for any  Hardy function $f \in H_0^2(\T)$, the function $H_\varphi(f)$ is real-analytic and we have 
\begin{align}\label{ineq_hankel_1}
\left\|\frac{H_\varphi (f)}{\sqrt{\varphi}}\right\|_{L^\infty(\T)} \le \| f\|_{L^2(\T; \varphi)}: =   \left(\int_\T |f(e^{it})|^2 \varphi (e^{it}) dm(t)\right)^{1/2}. 
\end{align}
\end{proposition}

For stating the next result, we introduce, for each integer $N\ge 1$,  an operator $E_N$ as follows: for $f$ a suitable function on $\T$, set 
\begin{align}\label{def_E_N}
E_N(f)(e^{i\theta}) : = \frac{1}{N} \sum_{k=0}^{N-1} f(e^{i  \frac{\theta +  2 \pi k}{N} }), \quad \theta \in \T.
\end{align}
Note that in the above definition of $E_N$, although each summand $f(e^{i \frac{\theta + 2 \pi k}{N}})$ depends on the choice of the representative in $\R$ of the point $\theta \in \T = \R/2 \pi \Z$, the whole sum does not. The operator $E_N$ is characterized by 
\begin{align}\label{Fourier_dilation}
\widehat{E_N(f)}(n) =  \widehat{f}(nN), \quad n\in \Z.
\end{align}

\begin{proposition}\label{prop_hankel_2}
Assume that $\varphi: \T\rightarrow \R_{+}$ is $2$-HP  and is not identically zero. Then for any  Hardy function $f \in H_0^2(\T)$ and any  analytic trigonometric polynomial $B_0$ of degree at most $N-1$, not identically zero:
\[
B_0 (e^{it}) = \sum_{n=0}^{N-1} b_n e^{i n t},
\]
we have
\begin{align}\label{ineq_hankel_2}
\left\| \frac{ E_N[ H_\varphi (B_0 f)]}{ \sqrt{E_N [ |B_0|^2\varphi ]}}  \right\|_{L^\infty(\T)} \le   \| f\|_{L^2(\T; \varphi)}. 
\end{align}
\end{proposition}

\begin{remark}
In Proposition \ref{prop_hankel_2},  the assumption that the non-zero trigonometric polynomial $B_0$ is of degree at most $N-1$ implies that $E_N[|B_0|^2 \varphi]$ never vanishes  on $\T$.  In general, the following inequality is not true: 
\[
\left\| \frac{ H_\varphi (B_0 f)}{ \sqrt{ |B_0|^2\varphi }}  \right\|_{L^\infty(\T)} \le   \| f\|_{L^2(\T; \varphi)}.
\]
\end{remark}

\begin{proposition}\label{prop_hankel_3}
Assume that $\phi$ is a function not identically zero and is of the form $\phi = P_{\sqrt{2/3}}*\mu$ for a positive Radon measure $\mu$ on $\T$.  Then for any  $f \in H_0^2(\T)$, we have
\begin{align}\label{ineq_hankel_3}
\left\| \frac{ P_{1/\sqrt{2}} * H_\phi (f) }{\sqrt{P_{1/\sqrt{2}} * \phi }} \right\|_{L^\infty(\T)} \le   \| f\|_{L^2(\T; \phi)}.  
\end{align} 
\end{proposition}

\subsection{Generalization of the Hankel inequalities}
The inequality \eqref{ineq_hankel_1} has natural generalizations. For more generalizations below, it is more convenient to work with the Riesz projection $R_{+}: L^2(\T) \rightarrow H^2(\T)$ (instead of using $R_{-}$).  Note that our definitions of the projections $R_{+}$ and $R_{-}$ are slightly different from the classical ones, in particular,  we have
\[
R_{+}  + R_{-} = Id + P_0,
\]
where $P_0$ is the projection onto the one-dimensional space of  constant functions on $\T$.

\begin{proposition}\label{prop_hankel_g1}
Given any $0\le r <1$ and any positive Radon measure $\mu$ on $\T$ such that $\mu(\T) \ne 0$. Let $\varphi = P_r* \mu$. Then for any $f\in H_0^2(\T)$, we have 
\begin{align}\label{ineq_c_r}
\left\|   \frac{H_\varphi(f)}{\sqrt{\varphi}}\right\|_{L^\infty(\T)} = \left\|   \frac{R_{+}(\bar{f} \varphi )}{\sqrt{\varphi}}\right\|_{L^\infty(\T)} \le \frac{r}{\sqrt{1-r^2}}  \| f\|_{L^2(\T; \varphi)}. 
\end{align}
The constant $r/\sqrt{1-r^2}$ in the above inequality is optimal. In particular, when $r  = 1/\sqrt{2}$, the inequality \eqref{ineq_c_r} reduces to the inequality \eqref{ineq_hankel_1}.
\end{proposition}


It seems natural to ask 
\begin{question}
For which class of functions $\varphi$ on $\T$ do we have $H_\varphi \in B(H^2, L^\infty)$, that is,  when does the Hankel operator $H_\varphi$ define a bounded operator from $H^2(\T)$ to $L^\infty(\T)$ ?
\end{question} 
  We will obtain a necessary condition and a sufficient condition for $H_\varphi \in B(H^2, L^\infty)$. Recall that for any $s> 0$,  the Sobolev space $\BH^{s}(\T)$ is defined as the set of functions $f: \T\rightarrow \C$ such that 
\[
\| f\|_{\BH^{s}}^2 =  \sum_{n\in \Z} (1+ n^2)^{s} \cdot | \widehat{f}(n)|^2<\infty.
\] 
 The Wiener space $\BA_1(\T) \subset L^\infty(\T)$ is defined as the set of functions $f: \T\rightarrow \C$ such that 
\[
\| f\|_{\BA_1}:  = \sum_{n\in \Z} | \widehat{f}(n)|<\infty.  
\]

\begin{proposition}\label{prop_pos_coef}
If $\varphi$ is a symbol such that all the Fourier coefficients of $R_{-}(\varphi)$ are positive. Then a necessary and sufficient condition for $H_\varphi \in B(H^2, L^\infty)$ is 
\begin{align}\label{pos_coef_cond}
\sum_{n=0}^\infty \Big(  \sum_{m =-\infty}^{-n} \widehat{\varphi}(m) \Big)^2 <\infty. 
\end{align}
\end{proposition}

\begin{proposition}\label{prop_2_inf}
 A necessary condition for  $H_\varphi \in B(H^2, L^\infty)$ is $R_{-}(\varphi) \in \BH^{\frac{1}{2}}(\T)$. A sufficient condition for  $H_\varphi \in B(H^2, \BA_1) \subset B(H^2, L^\infty)$ is $R_{-}(\varphi)\in  \BH^{1}(\T)$. 
\end{proposition}

The proof of the sufficient condition in  Proposition \ref{prop_2_inf} is based on the following Hilbert inequality or Hardy-Littlewood-P\'olya inequality  (cf. \cite[Theorem 341, p. 254]{Hardy-Littlewood-Polya}):
\begin{align}\label{HLP-ineq}
\Big| \sum_{m, n = 1}^\infty \frac{a_m b_n}{\max (m, n)} \Big| \le 4  \Big( \sum_{m=1}^\infty | a_m|^2\Big)^{1/2}\Big( \sum_{n=1}^\infty | b_n|^2\Big)^{1/2}. 
\end{align}

\section{HPD and HP functions}\label{sec_strongpdf}

 The main goal of this section is to prove Theorem \ref{thm_classification} on  the criterion of HPD functions on $\N$ and Theorem \ref{thm_class_bis}  on the criterion of HP functions on $\T$. 

\subsection{Basic properties for HPD functions}
Fix an integer $q\ge 2$ and set
\begin{align*}
\HH\mathcal{P}_\N(q): = \left\{ \alpha: \N \rightarrow \C \Big|  \text{$\alpha$ is  $q$-HPD function}\right\}. 
\end{align*}
 Clearly, from the definition of HPD functions, $\HH\mathcal{P}_\N(q)$ is a {\it closed positive cone}. That is, $\HH\mathcal{P}_\N(q)$ is closed under the operation of taking pointwise limit and the operation of taking  the linear combinations with positive-coefficients.    Other basic properties for HPD functions are obtained in the following Lemmas \ref{lem_mod} and \ref{lem_decay}.  

\begin{lemma}[Modulation invariance of HPD functions]\label{lem_mod}
Assume that $\alpha \in \HP_\N(q)$. Then for any $t \in \R$,  the function $\alpha_t: \N \rightarrow \C$ defined by 
\begin{align}\label{mod_inv}
\alpha_t(n) = e^{-i n t} \alpha(n), \quad n \in \N,
\end{align}
is again in $\HP_{\N}(q)$. Therefore, for any positive Radon measure $\nu$ on $\T$, the function $\alpha_\nu: \N \rightarrow \C$ defined by 
\begin{align}\label{Fourier_inv}
\alpha_\nu(n) = \widehat{\nu}(n) \alpha(n), \quad n \in \N,
\end{align}
is also in $\HP_\N(q)$.
\end{lemma}

\begin{lemma}\label{lem_decay}
Assume that $\alpha \in \HP_\N(q)$. Then the sequence $(q^{n/2} \alpha(n))_{n\in \N}$ is positive definite on $\N$. In particular, there exists a positive Radon measure $\mu$ on $\T$ such that  
\[
\alpha(n ) = q^{-n/2} \widehat{\mu}(n) =  q^{-n/2} \cdot\int_\T  e^{-i  n \theta} d \mu(\theta), \quad n \in \N.
\]
\end{lemma}

\begin{proof}[Proof of Lemma \ref{lem_mod}]
By Lemma \ref{lem-HPD}, there exists a centered stochastic process $(X_\sigma)_{\sigma \in \F_q^{+}}$ satisfying \eqref{GF_K}. 
Fix any $t\in \R$,  consider a  new stochastic process $(Z_\sigma)_{\sigma\in \F_q^{+}}$ defined by 
\[
Z_\sigma = e^{i | \sigma| t} X_\sigma, \quad \sigma \in \F_q^{+}. 
\]
A simple computation shows that 
\[
\Cov(Z_\sigma, Z_\delta) =  \TT_{\alpha_t}^{(\F_q^{+})} (\sigma, \delta) = \left\{
\begin{array}{ll}
e^{-i |w|t} \cdot \alpha(|w|) = \alpha_t(|w|); & \text{if $\delta = \sigma w$ for some $w \in \F_q^{+}$} \vspace{2mm}
\\
e^{i |w|t} \cdot \overline{\alpha(|w|)} = \overline{\alpha_t(|w|)}; & \text{if $\sigma = \delta w$ for some $w \in \F_q^{+}$} \vspace{2mm}
\\
0; & \text{otherwise}
\end{array}
\right..
\]
Note that for obtaining the above description of $\Cov(Z_\sigma, Z_\delta)$,  we have used the following two facts: 
\begin{itemize}
\item  the assumption  that $\TT_\alpha^{(\F_q^{+})}(\sigma, \delta) = 0$  if $\sigma, \delta$ are not comparable; 
\item if $\sigma, \delta$ are comparable, then if $\delta = \sigma w$, we have $| \delta|  = | \sigma| + |w|$ or if $\sigma = \delta w$, we have $|\sigma| = |\delta| + |w|$. 
\end{itemize}
Clearly, the equality $\TT_{\alpha_t}^{(\F_q^{+})}(\sigma, \delta) = \Cov(Z_\sigma, Z_\delta)$ for all $\sigma, \delta \in \F_q^{+}$ implies that $\TT_{\alpha_t}^{(\F_q^{+})}$ is a positive definite kernel on $\F_q^{+}$. Therefore, by definition, $\alpha_t$ is $q$-HPD. 

From the definition \eqref{Fourier_inv} of $\alpha_\nu$, we have 
\[
\alpha_\nu = \int_\T \alpha_t d\nu(t).
\]
Therefore, the assertion $\alpha_\nu\in \HP_\N(q)$  follows immediately from the fact that $\alpha_t\in \HP_\N(q)$ for any $t\in \R$ and the fact that $\HP_\N(q)$ is a closed positive cone.
\end{proof}

For proving Lemma \ref{lem_decay}, we need to construct a new stochastic process on $\N$ as follows. Let $(X_\sigma)_{\sigma\in \F_q^{+}}$ be  a centered stochastic process  on $\F_q^{+}$ satisfying \eqref{GF_K}. Consider now a centered stochastic process $(\Theta_n)_{n\in \N}$ on $\N$ defined by 
\begin{align}\label{def_theta}
\Theta_n: = \frac{1}{q^{n/2}}\sum_{\sigma\in \F_q^{+} \atop |\sigma| = n} X_\sigma.
\end{align}

\begin{lemma}\label{lem_sp_av}
The centered stochastic process $(\Theta_n)_{n\in \N}$ defined by \eqref{def_theta} is stationary and 
\begin{align}\label{cov_Y}
\Cov(\Theta_n, \Theta_{n+k}) = q^{k/2}\alpha(k), \quad n, k \in \N.
\end{align}
\end{lemma}
\begin{proof}
Recall that if $\sigma, \delta$ are not comparable, then $\E(X_\sigma \overline{X_\delta}) = K_\alpha(\sigma, \delta) = 0$. For any $n, k \in \N$, we have 
\begin{align*}
 \Cov(\Theta_n, \Theta_{n+k})=&   \E(\Theta_n \overline{\Theta_{n+k}}) 
=   \frac{1}{q^{n/2}} \cdot \frac{1}{q^{(n+k)/2}}\sum_{\sigma, \delta \in \F_q^{+} \atop |\sigma| =n, |\delta| = n+k} \E (X_\sigma \overline{X_\delta})
\\
= &     \frac{1}{q^{n/2}} \cdot \frac{1}{q^{(n+k)/2}}\sum_{\sigma\in \F_q^{+} \atop |\sigma|=n} \sum_{w\in \F_q^{+}\atop |w| = k} \E(X_\sigma \overline{X_{\sigma w}}) 
=  \frac{1}{q^{n/2}} \cdot \frac{1}{q^{(n+k)/2}}\sum_{\sigma\in \F_q^{+} \atop |\sigma|=n} \sum_{w\in \F_q^{+}\atop |w| = k}  \alpha(k) 
\\
 = &  \frac{q^n \cdot q^k}{ q^{n/2} \cdot q^{(n+k)/2}} \alpha(k) = q^{k/2} \alpha(k). 
\end{align*}
This completes the proof of \eqref{cov_Y}. Note that the stationarity  follows from \eqref{cov_Y}. 
\end{proof}

\begin{proof}[Proof of Lemma \ref{lem_decay}]
Lemma \ref{lem_sp_av}, combined with the classical Herglotz-Bochner Theorem on positive definite functions on $\Z$ or on $\N$ (cf.  Bochner \cite[Theorem 3.2.3]{Bochner1955}), implies that there exists a positive Radon measure $\mu$ on $\T$ such that 
\[
q^{k/2} \alpha(k) = \widehat{\mu}(k) = \int_\T e^{-ik \theta} d\mu(\theta), \quad k \in \N. 
\]
This completes the  proof of Lemma \ref{lem_decay}. 
\end{proof}

\subsection{The first proof of Lemma \ref{prop_Bozej}: Extension of positive definite kernels}
Our first proof of Lemma \ref{prop_Bozej} relies on a beautiful result of Bo\.{z}ejko \cite{Bozejko-89} about the construction of positive definite kernels on union of two sets decribed as follows. 

Let $K_i: \Sigma_i \times \Sigma_i \rightarrow \C$ be a kernel on a set $\Sigma_i$ ($i = 1, 2$). Assume that the intersection $\Sigma_1 \cap \Sigma_2 = \{x_0\}$ is a singleton and $K_1(x_0, x_0) = K_2(x_0, x_0)  = 1$. The {\it Markov product} of $K_1$ and $K_2$, denoted by $K_1*_{x_0} K_2$,  is a kernel $K$ on the union $\Sigma_1 \cup \Sigma_2$, defined by 
\begin{itemize}
\item $K|_{\Sigma_i\times \Sigma_i} = K_i ( i = 1, 2);$
\item (Markov property) For $\sigma_i \in \Sigma_i(i = 1, 2)$, 
\[
K(\sigma_1, \sigma_2)= K_1(\sigma_1, x_0) K_2(x_0, \sigma_2),  \quad K(\sigma_2, \sigma_1)= \overline{K(\sigma_1, \sigma_2)}.
\]
\end{itemize}

\begin{theorem}[{Bo\.{z}ejko \cite[Theorem 4.1]{Bozejko-89}}]\label{thm_Bozej}
Let $\Sigma_1, \Sigma_2$ be two sets such that the intersection $\Sigma_1 \cap \Sigma_2 = \{x_0\}$ is a singleton. 
Let  $K_1, K_2$ be two positive definite kernels on $\Sigma_1, \Sigma_2$  respectively. Assume that 
\[
K_i (\sigma_i, \sigma_i) = 1 \quad \text{for all $\sigma_i \in \Sigma_i \, (i = 1, 2)$.}
\]
Then the Markov product $K_1*_{x_0}K_2$ of $K_1$ and $K_2$ is also a positive definite kernel. 
\end{theorem}

\begin{proof}[Proof of Lemma \ref{prop_Bozej}]
We will divide the proof into several steps. We will deal with many restriction of the kernel $\TT_{\beta_q}^{(\F_q^{+})}$ onto subsets of $\F_q^{+}$. For simplifying notation, for any subset $\Sigma\subset \F_q^{+}$, we will denote 
\[
K(\Sigma): = \TT_{\beta_q}^{(\F_q^{+})}|_{\Sigma\times \Sigma}. 
\]

For any $\sigma, \delta \in \F_q^{+}$, we denote $d(\sigma, \delta)$ the distance of between  $\sigma$ and $\delta$ in the Cayley graph of $\F_q^{+}$  and write  $R$ the set of pairs $(\sigma, \delta)\in \F_q^{+}\times \F_q^{+}$ such that  $\sigma, \delta$ are comparable.   Then  by the definition of $\TT_{\beta_q}^{(\F_q^{+})}$ given as in \eqref{def_kernel_a}, we have 
\begin{align}\label{formula_R}
\TT_{\beta_q}^{(\F_q^{+})}(\sigma, \delta) = q^{-d(\sigma, \delta)/2} \cdot \ch_R(\sigma, \delta) =  q^{- d(\sigma, \delta)/2} \mathds{1}(\text{$\sigma$ and $\delta$ are comparable}).
\end{align}

\medskip

The following figure shows the steps of our proof. 

\begin{center}

\begin{tikzpicture}

\filldraw(-12,10) circle (2pt);
\node (e) at (-12-0.2,10+0.1) {$e$} ;
\filldraw(-11.5,9) circle (2pt);
\node (e) at (-11.5+0.2,9+0.1) {$s_1$} ;
\node (e) at (-12.5-0.3,9+0.1) {$s_2$} ;
\node (e) at (-12,8.2) {$\Sigma_e$} ;
\filldraw(-12.5,9) circle (2pt);
\draw (-12,10)--(-11.5,9);
\draw (-12,10)--(-12.5,9);
\node at (-12,10.6) {Step 1 ($q=2$)};

\filldraw(-9,10) circle (2pt);
\node(e) at (-9-0.2,10+0.1) {$e$};
\filldraw(-9.5,9) circle (2pt);
\node(e) at (-9.5-0.2,9+0.1) {$s_1$};
\filldraw(-8.5,9) circle (2pt);
\node(e) at (-8.5+0.2,9+0.1) {$s_2$};
\filldraw(-9,8) circle (2pt);
\node(e) at (-9+0.4,8+0.1) {$s_1^2$};
\filldraw(-10,8) circle (2pt);
\node(e) at (-10-0.4,8+0.1) {$s_1s_2$};
\node (e) at (-9,7.3) {$\Sigma_e\cup \Sigma_{s_1}$} ;
\draw (-9,10)--(-9.5,9);
\draw (-9,10)--(-8.5,9);
\draw [color=red!60, fill=red!5, thick] (-9.5,9)--(-9,8);
\draw [color=red!60, fill=red!5, thick] (-9.5,9)--(-10,8);
\node at (-9,10.6) {Step 2 ($q=2$)};

\filldraw(-6,10) circle (2pt);
\draw(-6,10)--(-6.5,9);
\draw(-6,10)--(-5.5,9);
\filldraw(-6.5,9) circle (2pt);
\draw(-6.5,9)--(-6.8,8);
\draw(-6.5,9)--(-6.2,8);
\filldraw(-5.5,9) circle (2pt);
\draw[color=red!60, fill=red!5, thick](-5.5,9)--(-5.8,8);
\draw[color=red!60, fill=red!5, thick](-5.5,9)--(-5.2,8);
\filldraw(-5.8,8) circle (2pt);
\filldraw(-5.2,8) circle (2pt);

\filldraw(-6.8,8) circle (2pt);
\filldraw(-6.2,8) circle (2pt);

\filldraw(-3.5,10) circle (2pt);
\draw(-3.5,10)--(-3,9);
\draw(-3.5,10)--(-4,9);
\filldraw(-3,9) circle (2pt);
\draw(-3,9)--(-3.3,8);
\draw(-3,9)--(-2.7,8);
\filldraw(-4,9) circle (2pt);
\draw(-4,9)--(-4.3,8);
\draw(-4,9)--(-3.7,8);
\filldraw(-3.3,8) circle (2pt);
\filldraw(-2.7,8) circle (2pt);
\filldraw(-4.3,8) circle (2pt);
\filldraw(-3.7,8) circle (2pt);
\draw[color=red!60, fill=red!5, thick](-4.3,8)--(-4.5,7);
\draw[color=red!60, fill=red!5, thick](-4.3,8)--(-4.1,7);
\filldraw(-4.5,7) circle (2pt);
\filldraw(-4.1,7) circle (2pt);

\node (e) at (-6,7.3) {$\Sigma_e\cup \Sigma_{s_1}\cup\Sigma_{s_2}$} ;
\node (e) at (-3-0.3,6.5) {$\Sigma_e\cup \Sigma_{s_1}\cup\Sigma_{s_2}\cup \Sigma_{s_1^2}$} ;
\node at (-4.5,10.6) {Step 3 ($q=2$)};

\node at (-0.8,9) {$\cdots\cdots$};
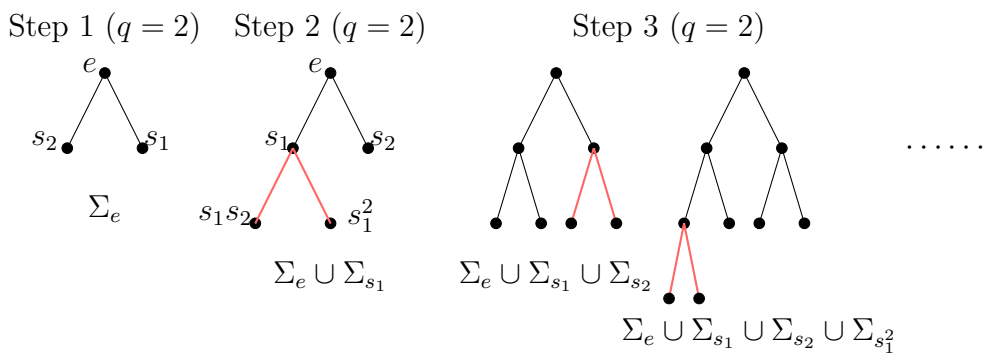
\node at (-6.5, 5.9) {Figure 3: steps of the proof of Lemma \ref{prop_Bozej}};
\end{tikzpicture}
\end{center}

{\flushleft \it Step 1.} We first show that for any $w\in \F_q^{+}$, the kernel $K(\Sigma_w)$ obtained as the restriction of $\TT_{\beta_q}^{(\F_q^{+})}$ on the $(q+1)$-element subset 
\[
\Sigma_w : = \{w, ws_1, ws_2, \cdots, ws_q\}
\]
 is positive definite. Indeed, when $\Sigma_w$ is ordered in the above way, the restricted kernel $K(\Sigma_w)$ is a $(q+1)\times (q+1)$-matrix given by 
\[
A_q = \left[
\begin{array}{cc}
1 &v_q\\
v_q^t&  I_q
\end{array}
\right],
\]
where $I_q$ is the identity matrix of size $q\times q$ and $v_q = (q^{-1/2}, q^{-1/2}, \cdots, q^{-1/2})$ is the $(1 \times q)$-row vector with all coefficients equal $q^{-1/2}$, and $v_q^t$ is the transpose of $v_q$. A simple computation shows that 
\[
A_q = \left[
\begin{array}{cc}
1 &v_q\\
v_q^t&  I_q
\end{array}
\right] =   \left[
\begin{array}{cc}
1 & v_q\\
0 &  I_q
\end{array}
\right]^t \cdot \left[
\begin{array}{cc}
1 &0\\
0&  I_q - v_q^t v_q
\end{array}
\right] \cdot\left[
\begin{array}{cc}
1 &v_q\\
0&  I_q
\end{array}
\right].
\]
Since $v_q^tv_q$ correpsonds to the orthogonal projection onto the one-dimensional subspace $\C v_q \subset \C^q$ spanned by $v_q$, the operator $I_q - v_q^t v_q$ is the orthogonal projection onto the subspace $v_q^\perp \subset \C^q$ of orthogonal complement of $\C v_q$. Therefore, $I_q - v_q^t v_q$ is positive definite and it follows that $A_q$ and thus $K(\Sigma_w)$ is positive definite.

\medskip

{\flushleft \it Step 2.} Note that $\Sigma_e \cap \Sigma_{s_1} = \{s_1\}$ is a singleton, and $K_{\beta_q} (\sigma, \sigma) = 1$ for all $\sigma \in \F_q^{+}$, therefore, we may apply Bo\.{z}ejko's result Theorem \ref{thm_Bozej}. First, we shall show that  the kernel $K(\Sigma_e \cup \Sigma_{s_1})$ is the Markov product of $K(\Sigma_e)$ and $K(\Sigma_{s_1})$ :
\begin{align}\label{markov_0}
K(\Sigma_e\cup \Sigma_{s_1}) = K(\Sigma_e)*_{s_1} K(\Sigma_{s_1}). 
\end{align}
Then by Theorem \ref{thm_Bozej}, the kernel $K(\Sigma_e \cup \Sigma_{s_1})$ is positive definite. 

Indeed, it suffices to verify that for any $\sigma_1 \in \Sigma_e$ and $\sigma_2 \in \Sigma_{s_1}$, we have 
\begin{align}\label{markov_1}
\TT_{\beta_q}^{(\F_q^{+})}(\sigma_1, \sigma_2)= \TT_{\beta_q}^{(\F_q^{+})}(\sigma_1, s_1) \TT_{\beta_q}^{(\F_q^{+})}(s_1, \sigma_2). 
\end{align}
Clearly, from the simple structure of $\Sigma_e$ and $\Sigma_{s_1}$, we have $(\sigma_1, \sigma_2)\in R$ if and only if $(\sigma_1, s_1) \in R$. Hence both sides of \eqref{markov_1} vanish when $(\sigma_1, \sigma_2)\notin R$. Now if $(\sigma_1, \sigma_2) \in R$, then the three elements $\sigma_1, s_1, \sigma_2$ (not necessarily distinct) are mutually comparable and are ordered $\sigma_1 \preceq s_1 \preceq \sigma_2$. Hence \eqref{markov_1} follows from the formula  \eqref{formula_R} and the fact 
\[
d(\sigma_1, \sigma_2) = d(\sigma_1, s_1) + d(s_1, \sigma_2). 
\] 

\medskip

{\flushleft \it Step 3.} The above argument in Step 2 can be repeated to show that 
\[
K(\Sigma_e \cup \Sigma_{s_1} \cup \Sigma_{s_2}) = K(\Sigma_e \cup \Sigma_{s_1})*_{s_2} K(\Sigma_{s_2})
\] and  the positive definiteness of the kernel $K(\Sigma_e \cup \Sigma_{s_1} \cup \Sigma_{s_2})$ then follows from the positive definiteness of $K(\Sigma_e \cup \Sigma_{s_1})$ and $K(\Sigma_{s_2})$  by applying Theorem \ref{thm_Bozej}. This argument can be continuously repeated for proving the positive definiteness of $\TT_{\beta_q}^{(\F_q^{+})}$. 

For clarity, let us explain our strategy in more details. We shall prove that for any subset $\Sigma \subset \F_q^{+}$ and $w\in \F_q^{+}$ such that 
\begin{itemize}
\item $e\in \Sigma$,
\item $\Sigma$ is connected in the Cayley graph of $\F_q^{+}$, 
\item $\Sigma \cap \Sigma_w = \{w\}$,
\end{itemize}
 we have the following Markov property
\begin{align}\label{general_markov}
K(\Sigma \cup \Sigma_w)= K(\Sigma)*_{w} K(\Sigma_w). 
\end{align}
After proving \eqref{general_markov}, for completing the proof, we only need to write  
\[
\F_q^{+} = \bigcup_{n = 1}^\infty \Sigma_n
\]
in such a way  that 
\begin{itemize}
\item $\Sigma_1 = \Sigma_e$,  
\item $\Sigma_{n+1} = \Sigma_{n} \cup \Sigma_{w_n}$,
\item  $\Sigma_n \cap \Sigma_{w_n} = \{w_n\}$ for a certain $w_n \in \F_q^{+}$
\item $\Sigma_n$ is connected for all $n$,
\end{itemize}
 and then use 
\[
K(\Sigma_{n+1})= K(\Sigma_n)*_{w_n} K(\Sigma_{w_n})
\] 
to prove the positive definiteness of $K(\Sigma_{n+1})$ from the positive definiteness of $K(\Sigma_n)$ and $K(\Sigma_{w_n})$. 

Now let us prove \eqref{general_markov}. We only need to show that for any $\sigma_1 \in \Sigma$ and $\sigma_2 \in \Sigma_{w}$, we have 
\begin{align}\label{markov_general_bis}
\TT_{\beta_q}^{(\F_q^{+})}(\sigma_1, \sigma_2)= \TT_{\beta_q}^{(\F_q^{+})}(\sigma_1, w) \TT_{\beta_q}^{(\F_q^{+})}(w, \sigma_2). 
\end{align}

Note that $\sigma_2$ has only $(q+1)$-choices.  If $\sigma_2 = w$, then both sides of \eqref{markov_general_bis} are $\TT_{\beta_q}^{(\F_q^{+})}(\sigma_1, w)$. 

Assume now  $\sigma_2 = w s_i$ for some $i = 1, 2, \cdots, q$. 
{\flushleft \bf Claim II:} Note that $(\sigma_1, w s_i) \in R$ if and only if $\sigma_1, w,  \sigma_2 = ws_i$ are mutually comparable and are ordered by $\sigma_1 \preceq w \preceq \sigma_2 = w s_i$.  

Indeed, since $e, \sigma_2= ws_i$ determines a unique geodesic ray, if $\sigma_1$ is not before or equal $w$ in this geodesic ray, then $\sigma_1$ can not be connected to $e$ without passing the point  $\sigma_2 \in \Sigma_w$, this would contradict the fact that $\Sigma$ is connected. 

Therefore, if $(\sigma_1, ws_i) \notin R$ then $(\sigma_1, w) \notin R$ and  both sides of \eqref{markov_general_bis} vanish;  if $(\sigma_1, w s_i) \in R$, then using Claim II, we have 
\[
d(\sigma_1, w s_i) = d(\sigma_1, w) +1
\]
Hence in this case, we have 
\[
\TT_{\beta_q}^{(\F_q^{+})}(\sigma_1, ws_i)= q^{-d(\sigma_1, ws_i)/2} =  q^{-d (\sigma_1, w)/2} q^{-1/2} =  \TT_{\beta_q}^{(\F_q^{+})}(\sigma_1, w) \TT_{\beta_q}^{(\F_q^{+})}(w, ws_i).
\] 
This completes the proof of \eqref{general_markov} and thus the whole proof of Lemma \ref{prop_Bozej}.
\end{proof}

\subsection{The second proof of Lemma \ref{prop_Bozej}: Gaussian processes on $\F_q^{+}$}
Our second proof of Lemma \ref{prop_Bozej} relies on self-similarity of the Cayley graph of $\F_q^{+}$, that is, for any $\sigma \in \F_q^{+}$, the structure of $\sigma\cdot \F_q^{+}$ is the same as that of $\F_q^{+}$. 

Let $(G_\sigma)_{\sigma \in \F_q^{+}}$ be a family of i.i.d.  Gaussian real random variables, all of which have expectation $0$ and variance $1$. For any $r\in (0, 1)$, we construct a Gaussian process on $\F_q^{+}$ as follows. For each $\sigma \in \F_q^{+}$, define a random variable $X_\sigma^{(r)}$ as a linear averaging of our original i.i.d. Gaussian process $(G_\sigma)_{\sigma \in \F_q^{+}}$ over the subset $\sigma \cdot \F_q^{+} \subset \F_q^{+}$ by 
\begin{align}\label{def-X-r}
X_\sigma^{(r)}: = \sum_{k=0}^\infty  \frac{r^k}{q^{k/2}}\sum_{\tau \in \F_q^{+}: | \tau | = k}   G_{\sigma \tau}.
\end{align}
Clearly, since $G_\sigma$'s are independent,  for any integer $k\ge 0$,  we have 
\begin{align}\label{av-norm-1}
\Big\|  \frac{1}{q^{k/2}}\sum_{\tau \in \F_q^{+}: | \tau | = k}   G_{\sigma \tau} \Big\|_2 = 1
\end{align}
and the series \eqref{def-X-r} converges in $L^2$-sense. 

\begin{lemma}\label{lem-X-r}
For any real number $r\in (0, 1)$, the Gaussian process $(X_\sigma^{(r)})_{\sigma \in \F_q^{+}}$ satisfies 
\[
\Cov\left(\frac{X_\sigma^{(r)}}{\| X_\sigma^{(r)}\|_2}, \frac{X_{\delta}^{(r)}}{\| X_\delta^{(r)}\|_2}\right) = r^{d(\sigma, \delta)} q^{- d(\sigma, \delta)/2} \mathds{1}(\text{$\sigma$ and $\delta$ are comparable}).
\]
\end{lemma}

\begin{proof}
First of all, for any $\sigma \in \F_q^{+}$, by mutual independence of the Gaussian random variables $G_\sigma$'s and \eqref{av-norm-1}, we have 
\begin{align}\label{norm-X}
\| X_\sigma^{(r)}\|_2^2 =  \sum_{k =0}^\infty r^{2k} = \frac{1}{1 - r^2}. 
\end{align}
Now let $\sigma, \delta \in \F_q^{+}$ be any pair of distinct elements. If $\sigma$ and $\delta$ are not comparable, then the two subsets $\sigma \cdot \F_q^{+}$ and $\delta\cdot \F_q^{+}$ are disjoint and therefore, recalling that  the Gaussian process \eqref{def-X-r} is a real-valued centered process, we have
\[
\Cov(X_\sigma^{(r)}, X_{\delta}^{(r)}) =  \E\Big(X_\sigma^{(r)} X_{\delta}^{(r)} \Big) = 0.
\] 
If $\sigma$ and $\delta$ are comparable. By symmetry, we may assume that $\sigma \preceq \delta$, then there exists a unique $w\in \F_q^{+}$ such that $\delta = \sigma w$ with $|w| = d(\sigma, \delta)$ and we have   
\begin{align}\label{cov-X-X}
\begin{split}
\Cov(X_\sigma^{(r)}, X_{\delta}^{(r)})&  =  \E\Big(X_\sigma^{(r)} X_{\delta}^{(r)} \Big) =  \E\Big(X_\sigma^{(r)} X_{\sigma w}^{(r)} \Big) 
\\
& =  \E \Big( \sum_{k=0}^\infty  \frac{r^k}{q^{k/2}} \sum_{\tau \in \F_q^{+}: | \tau | = k}   G_{\sigma \tau} \cdot \sum_{\ell=0}^\infty  \frac{r^\ell}{q^{\ell/2}}\sum_{\iota \in \F_q^{+}: | \iota | = \ell}   G_{\sigma w \iota}\Big)  
\\
 & = \sum_{k=0}^\infty \sum_{\ell = 0}^\infty \frac{r^{k  +\ell}}{ q^{\frac{k+\ell}{2}}}   \sum_{\iota \in \F_q^{+}: | \iota | = \ell}    \E\Big(  \sum_{\tau \in \F_q^{+}: | \tau | = k} G_{\sigma \tau} \cdot    G_{\sigma w \iota} \Big)  
\\
& = \sum_{k=0}^\infty \sum_{\ell = 0}^\infty \frac{r^{k  +\ell}}{ q^{\frac{k+\ell}{2}}}   \sum_{\iota \in \F_q^{+}: | \iota | = \ell}    \sum_{\tau \in \F_q^{+}: | \tau | = k}   \mathds{1}_{\{\tau = w \iota \}}  
\\
& = \sum_{k=0}^\infty \sum_{\ell = 0}^\infty \frac{r^{k  +\ell}}{ q^{\frac{k+\ell}{2}}} \mathds{1}_{\{k= \ell + |w|\}}  q^{\ell} 
\\
& = r^{|w|} q^{-|w|/2} \frac{1}{1-r^2}  = r^{d(\sigma, \delta)} q^{- d(\sigma, \delta)/2}  \frac{1}{1-r^2}.
\end{split}
\end{align}
Comparing \eqref{norm-X} and \eqref{cov-X-X}, we complete the proof of the lemma. 
\end{proof}

\begin{proof}[Proof of Lemma \ref{prop_Bozej}]
Lemma \ref{lem-X-r} implies that for any $r\in(0,1)$, the kernel $K^{(r)}: \F_q^{+}\times \F_q^{+}\rightarrow \R$ defined by 
\[
K^{(r)}(\sigma, \delta) = r^{d(\sigma, \delta)} q^{- d(\sigma, \delta)/2} \mathds{1}(\text{$\sigma$ and $\delta$ are comparable})
\]
is positive definite. Therefore, recalling \eqref{formula_R}, the kernel $\TT_{\beta_q}^{(\F_q^{+})}$, being  the  limit of  the positive definite kernels $K^{(r)}$ as  $r\to 1$:  
\[
\TT_{\beta_q}^{(\F_q^{+})}(\sigma, \delta) =  q^{- d(\sigma, \delta)/2} \mathds{1}(\text{$\sigma$ and $\delta$ are comparable}) = \lim_{r\to 1} K^{(r)}(\sigma, \delta), \quad \forall \sigma, \delta \in \F_q^{+}, 
\]
is also positive definite. 
\end{proof}

\subsection{The third proof of Lemma \ref{prop_Bozej}}
Note that the boundary $\partial T_q$  is naturally homeomorphic to the Cantor group $(\Z/ q\Z)^\N$. Let $\mu_q$ be Cantor probability measure on $\partial T_q$ corresponding to the normalized Haar measure  on $(\Z/ q\Z)^\N$.  For any $\sigma \in T_q$, define
\[
C(\sigma) = \Big\{\xi \in \partial T_q\Big| \sigma \in \xi\Big\},
\]
that is, $C(\sigma)$ is the subset of rooted geodesic rays passing through the vertex $\sigma \in T_q$. The measure $\mu_q$  is characterized by the following property: for any $\sigma \in T_q$, we have
\[
\mu_q(C(\sigma)) = \frac{1}{q^{|\sigma|}}.
\]
Now for each $\sigma \in T_q$, we define a function $f_\sigma: \partial T_q \rightarrow \R$ by
\[
f_\sigma: = \frac{\mathds{1}_{C(\sigma)}}{\sqrt{\mu_q(C(\sigma))}}
\]
Then immediately,  we obtain a positive definite kernel $K: T_q \times T_q \rightarrow \R$ defined by 
\begin{align}\label{def-K-inner}
K(\sigma, \delta) = \int_{\partial T_q} f_\sigma f_\delta d\mu_q. 
\end{align}
Note that, for any pair $(\sigma, \delta)$ of non-comparable vertices of $T_q$, the subsets $C(\sigma)$ and $C(\delta)$ are disjoint and hence $K(\sigma, \delta) = 0$ for all such pairs of vertices.  On the other hand, if $\sigma, \delta$ are comparable, say, we have $\sigma \pless \delta$, then $C(\sigma) \cap C(\delta) = C(\delta)$ and 
\[
K(\sigma, \delta)  = \frac{\mu_q(C(\tau))}{ \sqrt{\mu_q(C(\sigma))} \cdot \sqrt{\mu_q(C(\delta))}} =q^{\frac{|\sigma|-|\delta|}{2}} = q^{-d(\sigma, \delta)/2}. 
\]
Therefore, the positive definite kernel in \eqref{def-K-inner} has the following form:
\[ 
K(\sigma, \delta) =  q^{- d(\sigma, \delta)/2} \mathds{1}(\text{$\sigma$ and $\delta$ are comparable}),
\]
which, by using the identification between $T_q$ and $\F_q^{+}$,  is exactly the kernel $\TT_{\beta_q}^{(\F_q^{+})}$ defined in \eqref{def_kernel_a}. Hence $\TT_{\beta_q}^{(\F_q^{+})}$ is positive definite and we complete the proof of the lemma.

\subsection{Proof of Theorem \ref{thm_classification}}
By Lemma \ref{lem_decay},  to prove Theorem \ref{thm_classification}, it suffices to prove that for any positive Radon measure $\mu$ on $\T$, the function 
\[
\alpha(n): = q^{-n/2} \widehat{\mu}(n) = \beta_q(n) \widehat{\mu}(n), \quad n \in \N,
\] 
defines a $q$-HPD function. But clearly, Lemma \ref{prop_Bozej} and Lemma \ref{lem_mod} together immediately imply that the above function $\alpha$ is $q$-HPD.  

\subsection{Proof of Theorem \ref{thm_class_bis}}
Fix an integer $q\ge 2$, we set
\[
\HH\mathcal{P}_\T(q): = \left\{ \varphi: \T \rightarrow \R \Big|  \text{$\varphi$ is  $q$-HP}\right\}.
\]
 Clearly, from the definition of HP functions, $\HH\mathcal{P}_\T(q)$ is a closed positive cone and we have a natural affine correspondence between $\HP_\T(q)$ and $\HP_\N(q)$ given by: 
\begin{align}\label{p-pd-corr}
\HP_\T(q) \ni \varphi \mapsto \widehat{\varphi}\in \HP_\N(q).
\end{align}
This correspondence \eqref{p-pd-corr} and Theorem \ref{thm_classification}  imply immediately Theorem \ref{thm_class_bis}.

\subsection{Proof of Theorem \ref{thm-general-T}} If a rooted tree $T$ is of uniformly bounded valence, then $T$ can be embeded into a rooted homogeneous tree $T_q$. Now by restricting a non-trivial branching-type SSP on $T_q$ (whose existence is proved by Theorem \ref{thm_class_bis}, for instance, the branching-type SSP on $T_q$  corresponding to the function $\beta_q$ in Lemma \ref{prop_Bozej}) onto the sub-tree $T$, we obtain a non-trivial branching-type SSP on $T$. 

Conversely, we show the existence of non-trivial branching-type SSP on $T$ implies that $T$ is of uniformly bounded valence. We argue by contradiction, assume that $T$ is not of uniformly bounded valence.  Then  by the assumption that $T$ has no leaves, for any $n\ge 1$ and  any $q\ge 2$, there exist distinct $\sigma_0, \sigma_1, \cdots, \sigma_q\in T$ such that 
\begin{align}\label{conf-loc}
\sigma_0 \pless \sigma_i,  \quad d(\sigma_0, \sigma_i) = n,   \quad i = 1, \cdots, q.
\end{align}
Note that $\sigma_1, \sigma_2, \cdots, \sigma_q$ are mutually non-comparable.  Let now $\nu$ be the spectral measure of a mean-zero branching-type SSP $X = (X_\sigma)_{\sigma\in T}$ on $T$. Replacing $(X_\sigma)_{\sigma\in T}$ by $(\frac{X_\sigma}{\| X_\sigma\|_2})_{\sigma\in T}$ if necessary, we may assume that $\|X_\sigma\|_2  = 1$ for all $\sigma \in T$.  Then the covariance matrix of the random vector $(X_{\sigma_0}, X_{\sigma_1}, \cdots, X_{\sigma_q})$ is given by a non-negative definite matrix: 
\[
C = \left[
\begin{array}{cc}
1 &w_q\\
w_q^* &  I_q
\end{array}
\right],
\]
where $I_q$ is the identity matrix of size $q\times q$ and $w_q = (\widehat{\nu}(n), \widehat{\nu}(n), \cdots, \widehat{\nu}(n))$ is the $(1 \times q)$-row vector with all coefficients $\widehat{\nu}(n)$. A simple computation shows that 
\[
\left[
\begin{array}{cc}
1-w_q w_q^* &0\\
0&  I_q
\end{array}
\right]  =        \left[
\begin{array}{cc}
1 & -w_q\\
0 &  I_q
\end{array}
\right] \cdot  C \cdot\left[
\begin{array}{cc}
1 &-w_q\\
0&  I_q
\end{array}
\right]^*.
\]
Since $C$ is non-negative definite, we obtain that 
\[
1 - q | \widehat{\nu}(n)|^2 = 1 - w_q w_q^* \ge 0. 
\]
That is, we have 
\begin{align}\label{fourier-up}
| \widehat{\nu}(n)|^2\le 1/q.
\end{align}
 Now since $n\ge 1, q\ne 2$ are arbitrary, we obtain that 
\[
\widehat{\nu}(n) =0, \quad \forall n\ge 1. 
\]
This implies that the branching-type SSP $X$ is trivial.  But $X$ is chosen arbitrary, hence all branching-type SSP on $T$ is trivial. Thus we complete the whole proof.

\subsection{Proof of Proposition \ref{prop-ac}}
Assume that $\nu$ is the spectral measure of a branching-type SSP on $T$. 
By the definition of $\Delta_n(T)$, for any $n\ge 1$, we can find distinct vertices $\sigma_0, \sigma_1, \cdots, \sigma_{\Delta_n(T)} \in T$ such that 
\begin{align}\label{conf-loc-as}
\sigma_0 \pless \sigma_i,  \quad d(\sigma_0, \sigma_i) = n,   \quad i = 1, \cdots, \Delta_n(T).
\end{align}
Using the same derivation as in 
\[
\text{\eqref{conf-loc} $\Longrightarrow$ \eqref{fourier-up}},
\]
the existence of vertices $\sigma_0, \sigma_1, \cdots, \sigma_{\Delta_n(T)}\in T$ satisfying \eqref{conf-loc-as}  implies that the spectral measure $\nu$ satisfy 
\[
| \widehat{\nu}(n)|^2 \le \frac{1}{\Delta_n(T)}. 
\]
Therefore, the assumtion \eqref{inv-sum} implies that 
\[
\sum_{n\in \Z} | \widehat{\nu}(n)|^2<\infty.
\]
It follows that the function $\rho: \T \rightarrow \R_{+}$ defined by
\[
\rho(\theta)= \sum_{n\in \Z} \widehat{\nu}(n) e^{i n \theta} 
\]
is in $L^2(\T, m)$ and $\nu(d\theta) = \rho(\theta) dm(\theta)$. This completes the whole proof of Proposition \ref{prop-ac}.

\section{Proof of Theorem \ref{thm-2-1}}
We start with an elementary result on the non-negative definite matrices, whose proof is included as well for the reader's convenience.  

For a square matrix $A$, we write $A\ge 0$ if and only if $A$ is non-negative definite.  In what follows, elements in $\C^n$ will be considered as row vectors. 
\begin{lemma}\label{lem-pd} 
Assume that $a \ge 0$ and let $v\in \C^n$ be a row vector and $B$ be an $n\times n$ non-negative matrix. Then 
\[
C = \left[
\begin{array}{cc}
a & v
\\
v^* & B
\end{array}
\right] \ge 0 \Longleftrightarrow  a B - v^*v \ge 0. 
\]
\end{lemma}

\begin{proof}
By approximation, we may assume that $a > 0$. Then replacing $C$ by $\frac{1}{a} C$ if necessary, we may assume that $a = 1$. The result follows immediately by observing the following equality
\[
\left[
\begin{array}{cc}
1 & -v
\\
0 & I_n
\end{array}
\right]^*
\left[
\begin{array}{cc}
1 & v
\\
v^* & B
\end{array}
\right]     \left[
\begin{array}{cc}
1 & -v
\\
0 & I_n
\end{array}
\right] 
=
\left[
\begin{array}{cc}
1 & 0
\\
0 & B-v^*v
\end{array}
\right].
\]
\end{proof}

\begin{proof}[Proof of Theorem \ref{thm-2-1}]
Fix any integer $q \ge 2$. For simplifying our notation, we identify  $T(q; 1)$ with the subset 
\begin{align}\label{def-pi-q}
\Pi_q: = \bigcup_{i=1}^q\{s_i^k:  k \in \N\} \subset \F_q^{+}.
\end{align}
The rooted tree structure on $T(q; 1)$ coincides with the sub-rooted-tree structure of $\Pi_q$. Using the identification $T(q; 1) \simeq \Pi_q$, a necessary and sufficient condition for a positive Radon measure $\mu$ on $\T$ to be the specture measure of a branching-type SSP on $T(q; 1)$ is that the following kernel $\TT^{(\Pi_q)}: \Pi_q \times \Pi_q \rightarrow \C$ is positive definite: 
\begin{align}\label{def-ker-q-1}
\TT^{(\Pi_q)} (s_i^k, s_j^l) = \left\{
\begin{array}{ll}
\widehat{\mu}(l-k) \cdot \mathds{1}(i = j), & \text{if $l, k \ge 1$;}
\vspace{2mm}
\\
\widehat{\mu}(l-k), & \text{if $l = 0$ or $k =0$.}
\end{array}
\right.
\end{align}
For any integer $n\ge 1$, set
\[
\Pi_q^{(n)}: = \bigcup_{i = 1}^q \{s_i^k: 1 \le k \le n\}.
\]
Clearly, the positive definiteness of the kernel $\TT^{(\Pi_q)}$ is by definition equivalent to the positive definiteness of all the finite matrices 
\[
C_n: = \TT^{(\Pi_q)}|_{\Pi_q^{(n)} \times \Pi_q^{(n)}}, \quad n \ge 1. 
\]
For any integer $n\ge 1$, we define the Toeplitz matrix $\TT_n(\mu)$  by
\begin{align}\label{def-T-n}
\TT_n(\mu) := \Big[ \widehat{\mu}(j-i)\Big]_{1\le i, j \le n}. 
\end{align}
By ordering the elements in the set $\Pi_q^{(n)}$ as 
\[
(e, \underbrace{s_1, s_1^2, \cdots, s_1^n}_{\text{the branch from $s_1$}}, \underbrace{s_2, s_2^2, \cdots, s_2^n}_{\text{the branch from $s_2$}}, \cdots, \underbrace{s_q, s_q^2, \cdots, s_q^n}_{\text{the branch from $s_q$}})
\]
and using the definition of $\TT_n(\mu)$ in \eqref{def-T-n}, the matrix $C_n$ has the following block form
\begin{align}\label{def-C-n}
C_n = \left[
\begin{array}{ccccc}
\widehat{\mu}(0)&  v_n &  v_n & \cdots &  v_n
\\
v_n^* & \TT_n(\mu) & 0 & \cdots & 0
\\
v_n^* & 0 & \TT_n(\mu) &  \cdots & 0
\\
\vdots & \vdots  & \vdots & \ddots & \vdots
\\
v_n^* & 0 & 0 & \cdots & \TT_n(\mu)
\end{array}
\right],
\end{align}
where $v_n \in \C^n$ is a row vector given by 
\[
v_n = (\widehat{\mu}(1), \widehat{\mu}(2), \cdots, \widehat{\mu}(n)). 
\]
Now applying Lemma \ref{lem-pd} and noting that $\mu(\T) = \widehat{\mu}(0)$,  the positive definiteness condition of the kernel $\TT^{(\Pi_q)}$ is equivalent to the following condition 
\begin{align}\label{many-pd-c}
\mu(\T) \cdot \left[
\begin{array}{cccc}
 \TT_n(\mu) & 0 & \cdots & 0
\\
 0 & \TT_n(\mu) &  \cdots & 0
\\
 \vdots  & \vdots & \ddots & \vdots
\\
 0 & 0 & \cdots & \TT_n(\mu)
\end{array}
\right] -  (v_n, v_n, \cdots, v_n)^* (v_n, v_n, \cdots, v_n) \ge 0, \quad \forall n\ge 1.  
\end{align}
The condition \eqref{many-pd-c} can be rewritten as 
\begin{align}\label{many-pd-c-bis}
\mu(\T) \sum_{k= 1}^q x_k \TT_n(\mu) x_k^* - \Big| \sum_{k = 1}^q v_n x_k^*\Big|^2 \ge 0, \quad \text{$\forall n \ge 1$ and $\forall x_1, \cdots, x_q \in \C^n$.} 
\end{align}
Note that for any $a = (a_1, \cdots, a_n)\in \C^n$, we have 
\begin{align*}
a \TT_n(\mu)a^*   & = \sum_{1 \le k, l \le n } \widehat{\mu}(l-k) a_k \overline{a_l} 
\\
& = \int_\T   \sum_{1 \le k, l \le n } e^{-i 2 \pi (l-k) \theta} a_k \overline{a_l}  d\mu (\theta)  = \int_\T \Big| \sum_{ k = 1}^n a_k e^{i 2 \pi k \theta}\Big|^2 d\mu(\theta)
\end{align*}
and 
\[
 v_n a^* =  \Big\langle 1, \sum_{ k = 1}^n a_k e^{i 2 \pi k \theta}\Big\rangle_{L^2(\T, \mu)} ,
\]
where $1$ stands for the constant function $1 (\theta) \equiv 1$ on $\T$. Therefore, for any $n\ge 1$, by writing 
\[
H_0^2(\T,\mu)_{\le n}: = \Big\{\sum_{k=1}^n a_k e^{i 2\pi k \theta}\Big| a   = (a_1, \cdots, a_n)\in \C^n\Big\},
\]
the condition \eqref{many-pd-c-bis} is equivalent to 
\begin{align}\label{many-pd-c-tri}
\Big|  \sum_{i=1}^q \Big\langle 1,  \,  P_i \Big\rangle_{L^2(\T, \mu)} \Big|^2 \le \mu(\T)  \sum_{i  =1}^q\| P_i\|_{L^2(\T, \mu)}^2, \text{$\forall n\ge 1, \forall P_1, \cdots, P_q \in H_0^2(\T, \mu)_{\le n}$.}
\end{align}
Introduce the vector-valued function space $L^2(\T, \mu; \C^q)$ and the corresponding subspace $H_0^2(\T, \mu; \C^q)_{\le n}$ for all $n \ge 1$. Note that we have orthogonal decompositions: 
\[
L^2(\T, \mu; \C^q) = \underbrace{ L^2(\T, \mu) \oplus \cdots \oplus L^2(\T, \mu)}_{\text{$q$-summands}}
\]
and 
\[
H_0^2(\T, \mu; \C^q)_{\le n} = \underbrace{ H_0^2(\T,\mu)_{\le n} \oplus \cdots \oplus H_0^2(\T,\mu)_{\le n}}_{\text{$q$-summands}}.
\]
Set $\vec{\mathds{1}} = (1, 1, \cdots, 1) \in L^2(\T, \mu; \C^q)$. Then the condition \eqref{many-pd-c-tri} is equivalent to the condition that for all $n\ge 1$ and all $\vec{P} \in H_0^2(\T, \mu; \C^q)_{\le n}$, we have 
\begin{align}\label{in-pd-le-norm}
\Big|\langle  \vec{\mathds{1}},   \vec{P}\rangle_{L^2(\T, \mu; \C^q)}\Big| \le \sqrt{\mu(\T)} \| \vec{P}\|_{L^2(\T,\mu; \C^q)}. 
\end{align}
For fixed $n\ge 1$, by writing  the orthogonal decomposition of $\vec{\mathds{1}}$ with respect to the subspace $H_0^2(\T, \mu; \C^q)_{\le n}$, the above inequality  \eqref{in-pd-le-norm} holds for all $\vec{P}\in H_0^2(\T, \mu; \C^q)_{\le n}$ if and only if 
\[
\inf_{\vec{P}\in H_0^2(\T, \mu; \C^q)_{\le n}} \| \vec{\mathds{1}} - \vec{P}\|_{L^2(\T, \mu; \C^q)}^2 \ge    \| \vec{1}\|^2_{L^2(\T, \mu; \C^q)}  - \mu(\T) =  (q -1)\mu(\T). 
\]
Using the above orthogonal decompositions of $L^2(\T, \mu; \C^q)$ and $H_0^2(\T, \mu; \C^q)_{\le n}$, we have 
\[
\inf_{\vec{P}\in H_0^2(\T, \mu; \C^q)_{\le n}} \| \vec{\mathds{1}} - \vec{P}\|_{L^2(\T, \mu; \C^q)}^2 = q \inf_{P\in H_0^2(\T, \mu)_{\le n}} \| 1 - P\|_{L^2(\T, \mu)}^2
\]
Therefore, we show that the condition \eqref{many-pd-c-bis} is equivalent to 
\[
 q \inf_{n\ge 1} \inf_{P\in H_0^2(\T, \mu)_{\le n}} \| 1 - P\|_{L^2(\T, \mu)}^2  \ge (q -1)\mu(\T). 
\]
Finally, by the classical Szeg\"o First Theorem (cf. \cite[Theorem 4.1.1]{Nikolski-hardy}), we have 
\[
\inf_{n\ge 1} \inf_{P\in H_0^2(\T, \mu)_{\le n}} \| 1 - P\|_{L^2(\T, \mu)}^2 = \exp \left(\int_\T  \log  \Big( \frac{d\mu_{ac}}{dm}(\theta)\Big)  dm(\theta)\right),
\]
hence the condition \eqref{many-pd-c-bis} is equivalent to the condition 
\[
\exp \left(\int_\T  \log  \Big( \frac{d\mu_{ac}}{dm}(\theta)\Big)  dm(\theta)\right)\ge \Big(1 -\frac{1}{q}\Big)\mu(\T) 
\]
and we complete the whole proof. 
\end{proof}

\section{Prediction-Theory results}

Theorem \ref{thm-pred} is based on the following simple observation of the symmetries.

\begin{lemma}\label{lem-sym}
Let $\alpha: \N\rightarrow \C$ be a $q$-HPD function on $\N$ and let $(X_\sigma)_{\sigma\in \F_q^{+}}$ be  any branching-type SSP on $\F_q^{+}$ satisfying \eqref{GF_K}. Then for any positive integer $n\ge 1$,  the orthogonal projection of  $X_e$ onto the finite-dimensional space 
\[
\spann \big\{X_\sigma \big| \sigma \in \F_q^{+}, \, 1\le |\sigma|\le n \big\}
\]
is of the form 
\[
\sum_{k=1}^n a_k \sum_{\sigma: | \sigma| = k }  X_\sigma.
\]
\end{lemma}

\begin{proof}
Recall that the orthogonal projection of $X_e$ onto a space minimize the $L^2$-distance of $X_e$ and the vectors in this space. Therefore, for the first assertion, it suffices to show that for any family of complex numbers $(c_\sigma)$, we have 
\begin{align}\label{sym-optimal}
\Big\| X_e - \sum_{k=1}^n     \frac{   \sum_{\tau: |\tau| = k}   c_\tau }{q^k}     \sum_{\sigma: | \sigma|= k}    X_\sigma \Big \|_2 \le \Big\| X_e - \sum_{\sigma \in \F_q^{+}: 1 \le | \sigma | \le n} c_\sigma X_\sigma \Big\|_2.
\end{align}
Let us show \eqref{sym-optimal}.  First introduce an action of  $n$-tuples $\pi = (\pi_1, \cdots, \pi_n) \in \BS_q^n$ of permutations of $\{1, \cdots, q\}$  on the set $\B_n := \{ \sigma \in \F_q^{+}: |\sigma| \le n \}$ by the following: $\pi(e) = e$ and if $\sigma = s_{i_1} s_{i_2}\cdots s_{i_j}$ with $1\le j \le n$, then set 
\[
\pi(\sigma) = \pi (s_{i_1} s_{i_2}\cdots s_{i_j}) =  s_{\pi_1(i_1)} s_{\pi_2(i_2)} \cdots s_{\pi_j(i_j)}. 
\]
For any $\sigma, \delta \in \B_n$, clearly, $\sigma$ and $\delta$ are comparable if and only if $\pi(\sigma)$ and $\pi(\delta)$ are comparable for any $\pi \in \BS_q^n$. Moreover, we have the identity of  the graph distances $d(\sigma, \delta) = d(\pi(\sigma),\pi(\delta))$. Therefore, by   \eqref{GF_K} and the specific structure  \eqref{def_kernel_a} of $K_\alpha$, for any $\pi \in \BS_q^n$, we have  
\[
\Big\| X_e - \sum_{\sigma \in \F_q^{+}: 1 \le | \sigma | \le n} c_\sigma X_{\pi(\sigma)} \Big\|_2 = \Big\| X_e - \sum_{\sigma \in \F_q^{+}: 1 \le | \sigma | \le n} c_\sigma X_\sigma \Big\|_2.
\]
 It follows  that 
\begin{align}\label{sym-av}
\begin{split}
\Big\|    \frac{1}{\#(\BS_q^n)}  \sum_{\pi \in \BS_q^n} \Big(  X_{e} -  \sum_{\sigma \in \F_q^{+}: 1  \le | \sigma | \le n} c_\sigma X_{\pi(\sigma)}\Big) \Big\|_2  &  \le \frac{1}{\#(\BS_q^n)} \sum_{\pi \in \BS_q^n} \Big\| X_e - \sum_{\sigma \in \F_q^{+}: 1 \le | \sigma | \le n} c_\sigma X_\sigma \Big\|_2 
\\
& =\Big\| X_e - \sum_{\sigma \in \F_q^{+}: 1 \le | \sigma | \le n} c_\sigma X_\sigma \Big\|_2.
\end{split}
\end{align}
Note that 
\begin{align*}
\frac{1}{\#(\BS_q^n)}  \sum_{\pi \in \BS_q^n}   \sum_{\sigma \in \F_q^{+}: 1  \le | \sigma | \le n} c_\sigma X_{\pi(\sigma)} &  = \frac{1}{(q!)^n} \sum_{\pi \in \BS_q^n} \sum_{k = 1}^n \sum_{|\sigma| = k} c_\sigma X_{\pi(\sigma)}
\\
&  = \frac{1}{(q!)^n} \sum_{k = 1}^n   \sum_{| \tau | =k} X_\tau\sum_{ |\sigma| = k} c_\sigma   \sum_{\pi \in \BS_q^n }\mathds{1}_{\{\pi(\sigma) = \tau\}}
\\
& = \frac{1}{(q!)^n} \sum_{k = 1}^n   \sum_{| \tau | =k} X_\tau   \sum_{|\sigma|=k} c_\sigma   \cdot [(q-1)!]^k (q!)^{n-k}
\\
 & = \sum_{k = 1}^n   \frac{1}{q^k} \sum_{| \tau | =k} X_\tau   \sum_{|\sigma|=k} c_\sigma.
\end{align*}
Therefore, we obtain the equality
\[
\frac{1}{\#(\BS_q^n)}  \sum_{\pi \in \BS_q^n} \Big(  X_{e} -  \sum_{\sigma \in \F_q^{+}: 1  \le | \sigma | \le n} c_\sigma X_{\pi(\sigma)}\Big) = X_e  - \sum_{k=1}^n     \frac{   \sum_{\tau: |\tau| = k}   c_\tau }{q^k}     \sum_{\sigma: | \sigma|= k}    X_\sigma,
\]
which combined with \eqref{sym-av} immediately implies the desired inequality \eqref{sym-optimal}. 
\end{proof}

\begin{proof}[Proof of Theorem \ref{thm-pred}]
Recall the definition \eqref{def_theta} for the stochastic process $(\Theta_n)_{n\in \N}$. 
Lemma \ref{lem-sym} implies that
\begin{align}\label{to-szego}
d_{L^2} \left(X_e, \, \overline{\spann}^{L^2}\big\{X_\sigma:  \sigma\in \F_q^{+}\setminus \{e\}\big\} \right) 
=   d_{L^2} \left(  \Theta_0, \, \overline{\spann}^{L^2}\big\{    \Theta_k: k = 1, 2, \cdots \big\} \right).
\end{align}
By Lemma \ref{lem_sp_av} and the equality \eqref{def_mu_a}, the centered stochastic process $(\Theta_n)_{n\in \N}$  is stationary and 
\[
\Cov(\Theta_n, \Theta_{n+k}) = q^{k/2}\alpha(k) = \widehat{\nu}_{\alpha}(k), \quad n, k \in \N.
\]
Therefore, by the classical Szeg\"o First Theorem  and recalling the Lebesgue decomposition \eqref{leb-dec}, we have 
\begin{align}\label{cl-szego}
d_{L^2} \left(  \Theta_0, \, \overline{\spann}^{L^2}\big\{    \Theta_k: k = 1, 2, \cdots \big\} \right) = \exp \left(\frac{1}{2}\int_\T  \log w_\alpha(\theta) dm(\theta)  \right).
\end{align}
Combining \eqref{to-szego} and \eqref{cl-szego},  we obtain the desired equality \eqref{pred-all}. 
\end{proof}

\begin{proof}[Proof of Proposition \ref{prop-pred-T-q-1}]
Recall the identification of $T(q; 1)$ with $\Pi_q$ in \eqref{def-pi-q}.  For the branching-type SSP $(X_\sigma)_{\sigma \in \Pi_q}$, set 
\[
L(i): = \overline{\spann}^{L^2} \Big\{X_{s_i^n}: n = 1, 2, \cdots\Big\}, \quad 1\le i \le q. 
\] 
Note that the spaces $L(1), \cdots, L(q)$ are mutually orthogonal. Let $P_i(X_e)$ be the orthogonal projection of $X_e$ onto $L(i)$, then by symmetry, we have 
\[
\| P_1(X_e)\|_{2} = \| P_2(X_e)\|_{2} = \cdots = \| P_q(X_e)\|_{2}. 
\]
Since $L(1), \cdots, L(q)$ are mutually orthogonal, the orthogonal projection onto $L(1) \oplus \cdots \oplus L(q)$ of the random variable $X_e$ is given by $P_1(X_e) + \cdots P_q(X_e)$ and thus 
\begin{multline*}
d_{L^2} \left(X_e, \, L(1) \oplus \cdots \oplus L(q) \right)  =  \sqrt{\|X_e\|_2^2 - \sum_{i = 1}^q \| P_i (X_e)\|_2^2} = \sqrt{\|X_e\|_2^2 - q \| P_1 (X_e)\|_2^2} \\\
= \sqrt{\|X_e\|_2^2 - q \Big( \| X_e\|_2^2 -  [d_{L^2} (X_e, \, L(1) )]^2\Big) }.
\end{multline*}
The classical Szeg\"o First Theorem says that 
\[
d_{L^2} (X_e, \, L(1) ) = \exp\left(\frac{1}{2}\int_\T  \log  \Big( \frac{d\mu_{ac}}{dm}\Big) dm\right).
\]
Therefore, we obtain 
\[
d_{L^2} \left(X_e, \, L(1) \oplus \cdots \oplus L(q) \right) = \sqrt{ q \exp \left(\int_\T  \log  \Big( \frac{d\mu_{ac}}{dm}\Big) dm\right)  - (q-1) \mu(\T)}.
\]
This completes the proof of the equality \eqref{pred-T-q-1}.
\end{proof}

\section{Two-weight Hankel inequalities}\label{sec_hankel}

\subsection{How we find these inequalities ?}
Let us now describe informally how the two-weight inequalities of Hankel operators arise naturally in our setting.

Let $\alpha : \N \rightarrow \C$ be a $q$-HPD function  on $\N$ and let $\varphi: \T\rightarrow \R_{+}$ be the corresponding $q$-HP function on $\T$ by \eqref{p-pd-corr}. Assume that $\varphi$ is not identically zero. Let $(X_\sigma)_{\sigma\in \F_q^{+}}$ be a centered stochastic  process on $\F_q^{+}$ satisfying 
\eqref{GF_K}. We can now explain the stationary stochastic processes behind our inequalities \eqref{ineq_hankel_1}, \eqref{ineq_hankel_2}, \eqref{ineq_hankel_3} involving Hankel operators.   

\begin{itemize}
\item  {\it The stationary stochastic processes behind the inequality \eqref{ineq_hankel_1}.} 

Let $q=2$.  Take any sequence $(a_n)_{n\in \N}$ in $\C$ such that the following series of random variables 
\begin{align}\label{series_rv}
\sum_{n\in \N}a_n X_{s_2^n}
\end{align}
converges in $L^2$-sense (which is equivalent to require that $\sum_{n\in \N}|a_n|^2 <\infty$, see Lemma \ref{lem_cov} below).  Then we can construct a stationary stochastic process $(Y_k)_{k\in \N}$ on $\N$ as follows:
\begin{align}\label{def_Y}
Y_k: = \sum_{n\in \N} a_n X_{s_1^k s_2^n}, \quad k\in \N.
\end{align}

\medskip

\begin{tikzpicture}
\node (e) at (-12-0.25,0.1) {$e$} ;
\node (s1) at (-11-0.25,1+0.2) {$s_1$} ;
\node (s12) at (-10-0.25,2+0.2) {$s_1^2$} ;
\node (s1k) at (-9-0.25,3+0.2) {$s_1^3$} ;
\node (s2) at (-11-0.25,-1-0.2) {$s_2$} ;
\node (s22) at (-10-0.25,-2-0.2) {$s_2^2$} ;
\node (s23) at (-9-0.25,-3-0.2) {$s_2^3$} ;
\node (s1s2) at (-10-0.35,0-0.2) {$s_1s_2$} ;
\node (s1s22) at (-9-0.45,-1-0.2) {$s_1s_2^2$} ;
\node (s1s22) at (-8-0.45,-2-0.2) {$s_1s_2^3$} ;
\node (s12s2) at (-9-0.35,1-0.2) {$s_1^2s_2$} ;
\node (s12s22) at (-8-0.45,0-0.2) {$s_1^2s_2^2$} ;
\node (s12s22) at (-7-0.45,-1-0.2) {$s_1^2s_2^2$} ;
\node (s1ks2) at (-8-0.35,2-0.2) {$s_1^3 s_2$} ;
\node (s1ks22) at (-7-0.45,1-0.2) {$s_1^3 s_2^2$} ;
\node (s1ks22) at (-6-0.45,0-0.2) {$s_1^3 s_2^3$} ;
\filldraw (-12,0) circle (2pt);
\filldraw (-11,1) circle (2pt);
\filldraw (-10,2) circle (2pt);
\filldraw (-9,3) circle (2pt);
\filldraw (-11,-1) circle (2pt);
\filldraw (-10,-2) circle (2pt);
\filldraw (-10,0) circle (2pt);
\filldraw (-9,-1) circle (2pt);

\filldraw (-9,1) circle (2pt);
\filldraw (-8,0) circle (2pt);
\filldraw (-8,2) circle (2pt);
\filldraw (-7,1) circle (2pt);

\filldraw (-9,-3) circle (2pt);
\filldraw (-8,-2) circle (2pt);
\filldraw (-7,-1) circle (2pt);
\filldraw (-6,0) circle (2pt);

\filldraw (-4,-1) circle (2pt);
\node (e1) at (-4-0.25,-1+0.1) {$e$} ;
\filldraw (-3.75,-0.75) circle (2pt);
\filldraw (-3.5,-0.5) circle (2pt);
\filldraw (-3.25,-0.25) circle (2pt);
\filldraw (-3,0) circle (2pt);
\node (e1) at (-2.8-0.7,0+0.3) {$s_1^{N-1}$} ;
\filldraw (-2.7,0.3) circle (2pt);
\node (e1) at (-2.7-0.3,0.3+0.4) {$s_1^{N}$} ;
\filldraw (-2.5,0.5) circle (2pt);
\filldraw (-2.2,0.8) circle (2pt);
\filldraw (-2,-1) circle (2pt);
\node (e1) at (-2+0.3,-1+0.5) {$s_1^{N-1}s_2^2$} ;
\filldraw (-1.5, -1.5) circle (2pt);

\filldraw (-2.5,-0.5) circle (2pt);
\node (e1) at (-2.5+0.3,-0.5+0.5) {$s_1^{N-1}s_2$} ;
\filldraw (-2,1) circle (2pt);
\node (e1) at (-2-0.3,1+0.3) {$s_1^{2N-1}$} ;
\filldraw (-1.5, 0.5) circle (2pt);
\node (e1) at (-1.5+0.5,0.5+0.3) {$s_1^{2N-1}s_2$} ;
\filldraw (-1, 0) circle (2pt);
\node (e1) at (-1+0.5,0+0.3) {$s_1^{2N-1}s_2^2$} ;
\filldraw (-0.5, -0.5) circle (2pt);
\filldraw (-1,2) circle (2pt);
\node (e1) at (-1-0.5,2+0.3) {$s_1^{(k-1)N}$} ;
\filldraw (-0.5,2.5) circle (2pt);
\filldraw (-0.7,2.3) circle (2pt);
\filldraw (-0.3,2.7) circle (2pt);
\filldraw (-2,-1) circle (2pt);

\filldraw (0,3) circle (2pt);
\node (e1) at (0-0.5,3+0.3) {$s_1^{kN-1}$} ;
\filldraw (0.5, 2.5) circle (2pt);
\node (e1) at (0.5+0.5,2.5+0.3) {$s_1^{kN-1}s_2$} ;
\filldraw (1, 2) circle (2pt);
\node (e1) at (1+0.5,2+0.3) {$s_1^{kN-1}s_2^2$} ;
\filldraw (0.3,3.3) circle (2pt);
\node (e1) at (0.3,3.3+0.3) {$s_1^{kN}$} ;
\node at (-8.5,-4) {Figure 4};
\node at (-1,-4) {Figure 5};
\draw (-12,0) -- (-8,4);
\draw (-12,0) -- (-8.5,-3.5);
\draw (-11,1) -- (-7.5,-2.5);
\draw (-10,2) -- (-6.5,-1.5);
\draw (-9,3) -- (-5.5,-0.5);
\draw[color=red!60, fill=red!5, thick] (-4,-1) -- (-3,0);
\draw[color=red!60, fill=red!5, thick] (-3,0) -- (-1,-2);
\draw[color=blue!60, fill=blue!5, thick] (-3,0) -- (-2,1);
\draw[color=blue!60, fill=blue!5, thick] (-2,1) -- (0,-1);
\draw[dashed] (-2,1) -- (-1,2);
\draw (0,3) --(1,4);
\draw[color=red!60, fill=red!5, thick] (-1,2) -- (0,3);
\draw[color=red!60, fill=red!5, thick] (0,3) -- (2,1);
\end{tikzpicture}

The stationarity of $(Y_k)_{k\in\N}$ can be directly seen from the structure of $\TT_\alpha^{(\F_2^{+})}$ or can be obtained by direct computation. We thus obtain a positive definite  sequence $(\Cov(Y_0, Y_k))_{k\in\N}$ on $\N$. If we know  further that 
\begin{align}\label{cov_sum}
\sum_{k\in\N} | \Cov(Y_0, Y_k)|<\infty,
\end{align}
then by Herglotz-Bochner Theorem on positive definite functions on $\N$, we will arrive at the following inequality
\begin{align}\label{Bochner_pos}
\Cov(Y_0, Y_0) + 2 \Re \Big(\sum_{k\ge 1} \Cov(Y_0, Y_k) e^{i k \theta} \Big) \ge 0, \quad \forall \theta \in \T. 
\end{align}
We can then derive a two-weight inequality in Proposition \ref{prop_hankel_1} for Hankel operators from the positivity condition \eqref{Bochner_pos}. 

\bigskip

\item {\it The stationary stochastic processes behind the inequality \eqref{ineq_hankel_2}. }

Let $q = 2$ and fix an integer $N\ge 1$.  We will construct a stationary process $(W_k)_{k\in \N}$ on $\N$ by setting 
\begin{align}\label{def_W}
W_k: = \sum_{n =0}^{N-1} a_n X_{s_1^{kN+n}} + \sum_{n= N}^\infty  a_n X_{s_1^{kN + N-1} s_2^{n-N+1}}, \quad k \in \N, 
\end{align}
where the sequence $(a_n)_{n\in \N}$ is chosen such that the series \eqref{def_W} converges in $L^2$-sense (which is again equivalent to require that $\sum_{n\in \N} |a_n|^2 <\infty$). In other words, $W_k$ is the linear-averaging along the infinite geodesic starting from $s_1^{kN}$ given below:
\[
(s_1^{kN}, s_1^{kN  + 1}, s_1^{kN +2}, \dots, s_1^{kN +N-1}, s_1^{kN+N-1} s_2, s_1^{kN+N-1} s_2^2, \cdots, s_1^{kN+N-1} s_2^{t}, \dots) 
\] 
Therefore, similarly as above,  using the stationarity of $(W_k)_{k\in \N}$ and Herglotz-Bochner Theorem, we obtain a similar inequality as \eqref{Bochner_pos}, we can derive the inequality \eqref{ineq_hankel_2}. 

\bigskip

\item {\it The stationary stochastic processes behind the inequality \eqref{ineq_hankel_3}.}

This time, we need to take $q  =3$ and fix an identification of $\F_2^{+}$ with the sub-semi-group of $\F_3^{+}$  as follows: 
\begin{align}\label{F_2_sub}
\F_3^{+} = \langle s_1, s_2, s_3\rangle, \quad \F_2^{+} = \langle s_2, s_3\rangle.
\end{align}
 That is, we let $\F_3^{+}$ be generated by $s_1, s_2, s_3$ and let $\F_2^{+}$ be the sub-semi-group  generated by $s_2, s_3$. We will construct a stationary process $(U_k)_{k\in \N}$ on $\N$ by setting
\begin{align}\label{def_U}
U_k: =   \sum_{n\in \N}  a_n   \cdot  \Big( \frac{1}{2^{n/2}} \sum_{\sigma\in \F_2^{+} \atop |\sigma| =n}   X_{s_1^k \sigma} \Big), \quad k \in \N.
\end{align}
Here again $(a_n)_{n\in \N}$ is a sequence in $\C$ chosen such that the above series converges in $L^2$-sense. Then the stationarity of $(U_k)_{k\in \N}$ and Herglotz-Bochner Theorem will lead to the inequality \eqref{ineq_hankel_3}.

\begin{center}

\begin{tikzpicture}[level/.style={sibling distance = 1cm/#1,  level distance = .5cm}, ]
\tikzset{
  treenode/.style = {align=center,circle,draw=black, inner sep=0pt, text centered,  fill=black , minimum width=0.2em, },}

\node (e) at (-0.25,0.1) {$e$};
\node (s1) at (3-0.25,0.25) {$s_1$} ;
\node (s1k) at (7-0.25,0.3) {$s_1^k$} ;

\node at (-0.3,-2) {$F_2^+=\langle s_2,s_3\rangle$}; 
\node at (3.3,-2) {$s_1\cdot F_2^+=s_1\cdot \langle s_2,s_3\rangle$};
\node at (7.5,-2) {$s_1^k\cdot F_2^+=s_1^k\cdot \langle s_2,s_3\rangle$};
\node at (4.5,-3) {Figure 6};
\node [treenode] at (0,0) {}
child{ node [treenode] {}
    child{ node [treenode] {}
        child{ node [treenode] {}
        }
        child{ node [treenode] {}
        }
    }
    child{ node [treenode] {}
        child{ node [treenode] {}
        }
        child{ node [treenode] {}
        }
    }
}
child{ node [treenode] {}
    child{ node [treenode] {}
        child{ node [treenode] {}
        }
        child{ node [treenode] {}
        }
    }
    child{ node [treenode] {}
        child{ node [treenode] {}
        }
        child{ node [treenode] {}
        }
    }
};

\node [treenode] at (3,0) {}
child{ node [treenode] {}
    child{ node [treenode] {}
        child{ node [treenode] {}
        }
        child{ node [treenode] {}
        }
    }
    child{ node [treenode] {}
        child{ node [treenode] {}
        }
        child{ node [treenode] {}
        }
    }
}
child{ node [treenode] {}
    child{ node [treenode] {}
        child{ node [treenode] {}
        }
        child{ node [treenode] {}
        }
    }
    child{ node [treenode] {}
        child{ node [treenode] {}
        }
        child{ node [treenode] {}
        }
    }
};

\node [treenode] at (7,0) {}
child{ node [treenode] {}
    child{ node [treenode] {}
        child{ node [treenode] {}
        }
        child{ node [treenode] {}
        }
    }
    child{ node [treenode] {}
        child{ node [treenode] {}
        }
        child{ node [treenode] {}
        }
    }
}
child{ node [treenode] {}
    child{ node [treenode] {}
        child{ node [treenode] {}
        }
        child{ node [treenode] {}
        }
    }
    child{ node [treenode] {}
        child{ node [treenode] {}
        }
        child{ node [treenode] {}
        }
    }
};

\filldraw (0,0) circle (2pt);
\filldraw (3,0) circle (2pt);
\filldraw (7,0) circle (2pt);

\draw (0,0) -- (3 ,0);
\draw [dashed](3,0) -- (8,0);
\end{tikzpicture}
\end{center}

\end{itemize}

\subsection{Proof of Proposition \ref{prop_hankel_1}}
\begin{lemma}\label{lem_ul_bd}
For any non-identically zero $q$-HP function  $\varphi: \T\rightarrow \R_{+}$, we have 
\[
\frac{\sqrt{q}-1}{ \sqrt{q}+1}\| \varphi\|_1 \le    \varphi(e^{i\theta}) \le   \frac{ \sqrt{q} + 1}{\sqrt{q}-1} \|\varphi\|_1, \quad \forall \theta \in \T.
\]
\end{lemma}
\begin{proof}
By Theorem \ref{thm_class_bis}, there exists a positive Radon measure $\mu$ on $\T$ such that \[
\varphi = P_{1/\sqrt{q}}*\mu (e^{i \theta})= \int_{\T} P_{1/\sqrt{q}} (e^{i(\theta-t)}) d\mu(t).
\] 
Therefore, 
\[
\varphi(e^{i \theta}) \ge \int_\T \inf_{t'\in \T} P_{1/\sqrt{q}} (e^{it'}) d\mu(e^{it}) =  \frac{\sqrt{q}-1}{\sqrt{q}+1} \int_\T d\mu = \frac{\sqrt{q}-1}{\sqrt{q} + 1} \| \varphi\|_1.
\]
The upper-bound can be obtained similarly. 
\end{proof}

\begin{lemma}\label{lem_cov}
For any $k\in \N$,  the linear map 
\[
\begin{array}{ccc}
H^2(\T)  & \longrightarrow &  \overline{\spann}^{L^2} (X_{s_1^k s_2^n}: n \in \N)
\vspace{2mm}
\\
 e^{in\theta} & \mapsto & X_{s_1^k s_2^n}
\end{array}
\]
is an isomorphism (not necessarily isometric). 
In particular,  the series  \eqref{series_rv} of random variables converges in $L^2$-sense if and only if $\sum_{n\in \N} |a_n|^2<\infty$.  
\end{lemma}
\begin{proof}
By the definition, for any $k\in \N$, the positive real-analytic function $\varphi$ is the spectral density of $(X_{s_1^k s_2^n})_{n\in \N}$ and we have
\begin{align}\label{eq_spectral}
 \Big\| \sum_{n\in \N} a_n X_{s_1^k s_2^n}\Big\|_{L^2}^2 = \int_{\T} \Big|  \sum_{n\in \N} a_n e^{in t}\Big|^2 \varphi(e^{it })\frac{d t}{2\pi}.
\end{align}
Therefore, Lemma \ref{lem_cov} is an immediate consequence of Lemma \ref{lem_ul_bd}. 
\end{proof}

\begin{lemma}\label{lem_cov_sum}
The inequality  \eqref{cov_sum} holds. 
\end{lemma}
\begin{proof}
Note that if $k< k'$ and $m \ge 1, n \ge 0$, then 
\begin{align}\label{orth_cond}
\E(X_{s_1^k s_2^m} \overline{X_{s_1^{k'} s_2^n}}) = 0.
\end{align}
Therefore,  for any $k, \ell \in \N$, if $\ell \ge 1$, then we have 
\begin{align}\label{cov_formula}
\Cov(Y_k, Y_{k+\ell}) = a_0 \sum_{n\in \N} \overline{a_n} \E(  X_{s_1^k} \overline{X_{s_1^{k+\ell} s_2^n}}) = a_0 \sum_{n\in \N} \overline{a_n} \widehat{\varphi}(n+\ell). 
\end{align}
By \eqref{exp_decay} and Cauchy-Schwarz inequality, for $\ell \ge 1$,  we have 
\begin{align}\label{cov_exp_decay}
\begin{split}
| \Cov(Y_k, Y_{k+\ell})| & \le |a_0| \Big(\sum_{n\in\N} |a_n|^2\Big)^{1/2}  \Big(  \sum_{n\in \N} q^{-(n+\ell)} \widehat{\varphi}(0)^2\Big)^{1/2} 
\\
& =  \frac{q^{-\ell/2} \widehat{\varphi}(0)|a_0|}{\sqrt{1-q^{-1}}} \Big(\sum_{n\in\N} |a_n|^2\Big)^{1/2}.
\end{split}
\end{align}
This inequality clearly implies \eqref{cov_sum}.
\end{proof}

\begin{proof}[Proof of Proposition \ref{prop_hankel_1}]
 For any $f\in H_0^2(\T)$, write $f(e^{i\theta}) =  \sum_{n=1}^\infty a_n e^{i n \theta}$.  Then by \eqref{cov_formula}, for any $\ell \ge 1$, we have 
\begin{align}\label{Fourier_coef}
\begin{split}
\Cov(Y_k, Y_{k+\ell}) & =  a_0 \sum_{n\in \N} \overline{a_n} \widehat{\varphi}(n+\ell) 
\\
&  = a_0 \int_\T  (\overline{a_0} + \overline{f (e^{it})}) \varphi(e^{it}) e^{-i \ell t} \frac{dt}{2\pi} 
\\
&  = |a_0|^2 \widehat{\varphi} (\ell) +  a_0 \overline{ (f\varphi)^{\wedge} (-\ell)},
\end{split}
\end{align}
where the integral should,  strictly speaking, be understood as inner product between two $L^2$-functions on $\T$; the symbol $(f\varphi)^{\wedge} (-\ell)$ means the $(-\ell)$-th Fourier coefficient of the function $f\varphi$. The inequalities \eqref{exp_decay} and \eqref{cov_exp_decay} imply  that $(f\varphi)^{\wedge} (-\ell)$ has an exponential decay for $\ell \rightarrow +\infty$ and thus we obtain a real-analytic function on $\T$:
\[
H_\varphi(f)(e^{i\theta}) = R_{-}(f\varphi)(e^{i\theta}) = \sum_{\ell \le 0} (f\varphi)^\wedge (\ell) e^{ i \ell \theta}.
\] 
The equality \eqref{eq_spectral} implies 
\begin{align}\label{2-norm}
\Cov(Y_k, Y_k)= \int_{\T} | a_0 +   f(e^{it})|^2 \varphi(e^{it})\frac{d t}{2\pi}.
\end{align}
  By substituting \eqref{Fourier_coef} and \eqref{2-norm} into \eqref{Bochner_pos}, we obtain  that for any $\theta \in \T$, 
\begin{multline}\label{pos_ineq}
0 \le \int_{\T} | a_0 +   f(e^{it })|^2 \varphi(e^{i t})\frac{d t}{2\pi} +  2 \Re  \Big[\sum_{\ell \ge 1}  \Big( |a_0|^2 \widehat{\varphi}(\ell) + a_0 \overline{  (f\varphi)^{\wedge}(-\ell)} \Big)e^{i \ell \theta}\Big] = 
\\
 = |a_0|^2 \widehat{\varphi}(0)  + 2 \Re  \Big(\overline{a_0}  \cdot (f\varphi)^\wedge (0)\Big) +  \int_\T |f(e^{it})|^2 \varphi(e^{it}) \frac{dt}{2 \pi} + 
\\
 + |a_0|^2 (\varphi(e^{i\theta}) - \widehat{\varphi}(0))  
+   2 \Re \Big[\overline{a_0}   R_{-}^{0} (f\varphi)  (e^{i\theta})\Big] = 
\\
  = |a_0|^2 \varphi(e^{i\theta}) + \int_\T |f (e^{it})|^2 \varphi(e^{it}) \frac{dt}{2\pi} + 2 \Re \Big[\overline{a_0}   R_{-} (f\varphi)  (e^{i\theta})\Big],
\end{multline}
where $R_{-}^0: L^2(\T)\longrightarrow L^2(\T)\ominus H^2(\T)$ denotes the orthogonal projection onto the subspace $L^2(\T)\ominus H^2(\T)$. 
Note that we can choose independently  $a_0$ and $f$.  Replacing $a_0$ by a suitable $a_0 e^{i t_\theta}$ (where $t_\theta \in \T$ can depend on $\theta$) in \eqref{pos_ineq}, we obtain that
\[
2 |a_0| | R_{-} (f\varphi) (e^{i\theta})| \le  |a_0|^2 \varphi(e^{i\theta}) + \int_\T |f (e^{it})|^2 \varphi(e^{it}) \frac{dt}{2\pi}.
\] 
Now since $a_0$ is arbitrary and independent of $f$, for any $f\in H_0^2(\T)$, we have 
\[
| H_\varphi (f)(e^{i\theta})| = |R_{-}(f\varphi)(e^{i \theta})| \le \sqrt{\varphi(e^{i\theta})} \cdot  \left(\int_\T |f(e^{it})|^2 \varphi(e^{it}) \frac{dt}{2\pi}\right)^{1/2}, \quad \forall \theta \in \T. 
\] 
This completes the proof of Proposition \ref{prop_hankel_1}.
\end{proof}

\subsection{Proof of Proposition \ref{prop_hankel_2}}
Fix an integer $N \ge 1$.  Take any sequence $(a_n)_{n\in \N}$ such that $\sum_{n\in\N} |a_n|^2 <\infty$.   Recall the definitions \eqref{def_Y} and \eqref{def_W} for $Y_k$ and $W_k$ respetively.  Using \eqref{eq_spectral},   we have 
\begin{align}\label{2_norm_W}
\|W_k\|_{L^2}^2 =  \| Y_k\|_{L^2}^2 = \int_\T  \Big| \sum_{n\in \N} a_n e^{int}\Big |^2 \varphi(e^{it})\frac{dt}{2\pi}. 
\end{align}
Thus, by Lemma \ref{lem_ul_bd}, the series \eqref{def_W} converges in $L^2$-sense if and only if   $\sum_{n\in\N}|a_n|^2<\infty$.

\begin{lemma}\label{lem_sta_W}
The stochastic process $(W_k)_{k\in\N}$ is stationary and the sequence $(\Cov(W_0, W_\ell))_{\ell \in\N}$ has an exponential decay as $\ell \to +\infty$. 
\end{lemma}

\begin{proof}
By \eqref{orth_cond}, the covariance between $W_k$ and $W_{k+\ell}$ for any $\ell \ge 1$ is computed by 
\begin{align}\label{cov_W}
\begin{split}
\Cov(W_k, W_{k+\ell})  =&     \sum_{n =0}^{N-1}  \sum_{n'=0}^{N-1} a_n   \overline{a_{n'}} \E \Big(  X_{s_1^{kN+n}} \overline{X_{s_1^{kN+ \ell N+ n'}}}\Big) 
 \\
& + \sum_{n =0}^{N-1}  \sum_{n'=N}^{\infty}  a_n   \overline{a_{n'}} \E \Big(  X_{s_1^{kN+n}} \overline{X_{s_1^{kN +  \ell N +N-1} s_2^{n'-N+1}}}\Big) 
\\
  =&  \sum_{n =0}^{N-1}  \sum_{n'=0}^{N-1} a_n   \overline{a_{n'}}  \widehat{\varphi}(\ell N + n' - n ) + \sum_{n =0}^{N-1}  \sum_{n'=N}^{\infty}  a_n   \overline{a_{n'}} \widehat{\varphi}(\ell N + n' -n)
\\
 = &  \sum_{n =0}^{N-1}  \sum_{n'=0}^{\infty}  a_n   \overline{a_{n'}} \widehat{\varphi}(\ell N + n' -n).
\end{split}
\end{align}
The equalities \eqref{2_norm_W} and \eqref{cov_W} together imply that $(W_k)_{k\in\N}$ is stationary. 

By \eqref{exp_decay}, for $\ell \ge 1$,  we have 
\begin{align*}
| \Cov(W_k, W_{k+\ell})| \le  \Big(\sum_{n=0}^{N-1} |a_n|\Big)  \Big(\sum_{n'\in\N} |a_n'|^2\Big)^{1/2} \frac{q^{-(\ell-1)N/2}}{\sqrt{1-q^{-1}}}  \widehat{\varphi}(0)^{1/2}.
\end{align*}
Thus the sequence $(\Cov(W_0, W_\ell))_{\ell \in\N}$ has an exponential decay as $\ell \to +\infty$. 
\end{proof}

\begin{proof}[Proof of Proposition \ref{prop_hankel_2}]
Lemma \ref{lem_sta_W} and Herglotz-Bochner Theorem for positive definite functions on $\N$ imply that 
\begin{align}\label{pos_cond_W}
\Cov(W_0, W_0) + 2 \Re \Big(  \sum_{\ell \ge 1}\Cov(W_0, W_\ell) e^{i \ell \theta}\Big) \ge 0, \quad \forall \theta \in \T. 
\end{align}
Define two functions on $\T$ by 
\[
A_0(e^{i\theta}):   = \sum_{n =0}^{N-1} a_n e^{i n \theta}, \quad F(e^{i\theta}): = \sum_{n= N}^{\infty} a_n e^{i n \theta}.
\]
The equalities \eqref{2_norm_W} and \eqref{cov_W} imply that 
\begin{align*}
\Cov(W_0, W_0) &= \int_\T  \Big| A_0(e^{it}) + F(e^{it}) |^2 \varphi(e^{it})\frac{dt}{2\pi};
\\
\Cov(W_0, W_\ell)& = \int_\T A_0(e^{it}) \Big[ \overline{A_0(e^{it})} + \overline{F(e^{it})}\Big] \varphi(e^{it})  e^{- i \ell N t}\frac{dt}{2\pi}.
\end{align*}
Then using \eqref{pos_cond_W} and by a similar computation as in \eqref{pos_ineq} and recalling \eqref{def_E_N} and \eqref{Fourier_dilation}, we obtain   
\begin{multline*}
0 \le  E_N \big[ |A_0|^2\varphi\big] (e^{i\theta})  +  \int_\T | F(e^{it})|^2 \varphi(e^{it}) \frac{dt}{2 \pi} + 2 \Re \Big(  \sum_{\ell =0}^\infty  \big(\overline{A_0} F \varphi \big)^{\wedge} (-\ell N) e^{-i \ell \theta}\Big) = 
\\
 = E_N \big[ |A_0|^2\varphi\big] (e^{i\theta})  +  \int_\T | F(e^{it})|^2 \varphi(e^{it}) \frac{dt}{2 \pi} + 2 \Re \Big(  E_N\Big[R_{-} (\overline{A_0} F\varphi)\Big]  (e^{i\theta})\Big).
\end{multline*}
Replacing $A_0(e^{i\theta})$ by a suitable $A_0(e^{i\theta})  \lambda e^{i t_\theta}$ with $\lambda \ge 0, t_\theta \in \T$, we obtain 
\[
2 \lambda \Big| E_N\big[R_{-} (\overline{A_0} F\varphi)\big]  (e^{i\theta})\Big| \le \lambda^2 E_N \big[ |A_0|^2\varphi\big] (e^{i\theta})  +  \int_\T | F(e^{it})|^2 \varphi(e^{it}) \frac{dt}{2 \pi}, \quad \forall \theta \in \T.
\]
By optimizing the above inequality, we obtain 
\begin{align}\label{ineq_A_F}
\Big| E_N\big[R_{-} (\overline{A_0} F\varphi)\big]  (e^{i\theta})\Big| \le  \Big( E_N \big[ |A_0|^2\varphi\big] (e^{i\theta})   \Big)^{1/2} \Big (  \int_\T | F(e^{it})|^2 \varphi(e^{it}) \frac{dt}{2 \pi}\Big)^{1/2},  \quad \forall \theta \in \T. 
\end{align}
Finally, let $f \in H_0^2(\T)$ and let $B_0$ be an analytic trigonometric polynomial of degree at most $N-1$, then by substituting  $F(e^{it})= e^{i(N-1)t} \cdot f(e^{it})$ and $A_0(e^{it}) = e^{i (N-1)t} \cdot \overline{B_0(e^{it})}$ into the inequality \eqref{ineq_A_F}, we obtain
\[
\Big| E_N\big[R_{-} (B_0 f\varphi)\big]  (e^{i\theta})\Big| \le  \Big( E_N \big[ |B_0|^2\varphi\big] (e^{i\theta})   \Big)^{1/2} \Big (  \int_\T | f(e^{it})|^2 \varphi(e^{it}) \frac{dt}{2 \pi}\Big)^{1/2},  \quad \forall \theta \in \T. 
\]
This completes the proof of the inequality \eqref{ineq_hankel_2}.
\end{proof}

\subsection{Proof of Proposition \ref{prop_hankel_3}}
Recall the definition \eqref{def_U} for $(U_k)_{k\in\N}$.
By assumption, the function $\varphi: \T\rightarrow \R_{+}$ is $3$-HP, by Theorem \ref{thm_class_bis}, there exists a positive measure $\mu$  on $\T$ such that $\varphi  = P_{1/\sqrt{3}} * \mu$. Set $\phi:  = P_{\sqrt{2/3}}*\mu$, thus $\varphi = P_{1/\sqrt{2}}*\phi$. We have 
\[
\widehat{\phi}(n)  =  2^{|n|/2} \widehat{\varphi}(n), \quad n \in \Z. 
\]
Using the equality \eqref{cov_Y}, for any $k\in \N$, we have 
\begin{align}\label{2_norm_U}
\Cov(U_k, U_k) =   \int_\T    \big|  \sum_{n\in \N} a_n   e^{i nt }\big|^2 \phi(e^{it}) \frac{dt}{2\pi}.
\end{align}
By Lemma \ref{lem_ul_bd} (applied to the non-identical-zero function $\phi$) and \eqref{2_norm_U}, the series \eqref{def_U} converges in $L^2$-sense if and only if $\sum_{n\in\N} |a_n|^2<\infty$. 

Set $f\in H_0^2(\T)$ by 
\[
f(e^{it})= \sum_{n=1}^\infty a_n e^{int}. 
\]  Recall our identification of $\F_2^{+}$ as a sub-free-semigroup of $\F_3^{+}$ given in \eqref{F_2_sub}. 
Note that for $k<k'$ and any $\sigma, \sigma'\in \F_2^{+} \subset \F_3^{+}$ such that $\sigma \ne e$, we have $\E(X_{s_1^k \sigma} \overline{X_{s_1^{k'} \sigma'}}) = 0$. Therefore, for any $k, \ell \in \N$ with $\ell \ge 1$, we have 
\begin{align}\label{cov_U}
\begin{split}
\Cov(U_k, U_{k + \ell}) & =   a_0    \sum_{n\in \N}    \overline{a_n}  \frac{1}{2^{n/2}}  \sum_{\sigma \in \F_2^{+} \atop |\sigma| = n} \E (X_{s_1^k}  \overline{X_{s_1^{k+\ell} \sigma}})   =   a_0 \sum_{n\in \N} \overline{a_n} 2^{n/2} \widehat{\varphi}(n + \ell)  
\\
& =  \frac{a_0}{2^{\ell/2}} \sum_{n\in \N} \overline{a_n} \widehat{\phi}(n + \ell) =  \frac{|a_0|^2}{2^{\ell/2}} \widehat{\phi}(\ell) + \frac{a_0}{2^{\ell/2}} \overline{ (f\phi)^\wedge(-\ell)}. 
\end{split}
\end{align}

The equalities \eqref{cov_U} and \eqref{2_norm_U}  imply that $(U_k)_{k\in\N}$ is stationary. And the equality \eqref{cov_U}, together with the fact $\phi = P_{\sqrt{2/3}}* \mu$, implies that 
\[
| \Cov(U_k, U_{k+\ell})| \le 3^{-\ell/2} \sqrt{3/2} |a_0|\Big(\sum_{n\in \N} |a_n|^2\Big)^{1/2}  \widehat{\phi}(0).
\]
Therefore, we have 
\begin{align}\label{pos_U}
\Cov(U_0, U_0) + 2 \Re \Big( \sum_{\ell \ge 1} \Cov( U_0, U_\ell) e^{i\ell \theta} \Big) \ge 0, \quad \forall \theta \in \T. 
\end{align}
Using similar argument as in the proofs of Propositions \eqref{prop_hankel_1} and \eqref{prop_hankel_2}, combining  \eqref{2_norm_U}, \eqref{cov_U} and \eqref{pos_U} and the fact that $\widehat{\phi}(0)  = \widehat{\varphi}(0)$,  we obtain  
\[
0\le |a_0|^2 \varphi(e^{i\theta}) + \int_\T |f(e^{it})|^2 \phi(e^{it}) \frac{dt}{2\pi} + 2 \Re \Big(\overline{a_0}  \big[ P_{1/\sqrt{2}} * R_{-} (f \phi) \big] (e^{i\theta})\Big).
\]
Again using the same arguments as in the proofs of Propositions \ref{prop_hankel_1} and \ref{prop_hankel_2}, we derive from the above inequality the following inequality 
\[
\Big| \big[P_{1/\sqrt{2}} * H_\phi (f) \big] (e^{i\theta})\Big| \le \sqrt{(P_{1/\sqrt{2}} * \phi) (e^{i\theta})}  \Big(\int_\T |f (e^{it})|^2 \phi(e^{it}) \frac{dt}{2\pi} \Big)^{1/2}.
\] 
This completes the whole proof of Proposition \ref{prop_hankel_3}.

\section{More general Hankel inequalities}

\begin{lemma}\label{lem_hankel_poisson}
Given any $0\le r <1$ and any $t\in \T$.  Then  for any $f\in H^2(\T)$ and any $\theta \in \T$,  we have 
\[
\left|  \frac{R_{-} (f P_{re^{it}}) (e^{i\theta})}{\sqrt{P_{r e^{it}}  (e^{i \theta})}} \right|   =\left|  \frac{R_{+} (\bar{f} P_{re^{it}}) (e^{i\theta})}{\sqrt{P_{r e^{it}}  (e^{i \theta})}} \right|  =   \frac{1}{\sqrt{1-r^2}} \left|\int_\T  f(e^{i \alpha}) P_{r e^{it}} (e^{i\alpha}) \frac{d\alpha}{2\pi}\right|.
\]
\end{lemma}

\begin{proof}
Fix $0\le r < 1$ and $t\in \T$.
Write $f(e^{i\theta}) = \sum_{n\in \N} a_n e^{in\theta} \in H^2(\T)$, we have 
\begin{align*}
& R_{-} (f P_{r e^{it}}) (e^{i\theta}) =  R_{-} \Big[ \sum_{n\in \N} a_n e^{in \theta } \sum_{m\in \Z} r^{|m|} e^{-imt} e^{i m \theta}\Big]  
\\
& =  \sum_{n\in \N} a_n \sum_{m = -\infty}^{-n} r^{|m|} e^{-imt} e^{i (m+n)\theta} 
= \sum_{n\in \N} a_n  e^{in\theta}\sum_{k = n}^\infty r^k e^{ikt} e^{-i k\theta}  = \frac{ \sum_{n\in \N} a_n   r^n e^{int} }{1- r e^{ik (t-\theta)}}.
\end{align*}
Therefore, by recalling \eqref{def_Poisson}, we obtain 
\begin{align}\label{id_poisson}
\left|  \frac{R_{-} (f P_{re^{it}}) (e^{i\theta})}{\sqrt{P_{r e^{it}}  (e^{i \theta})}} \right|    = \left|\frac{\sum_{n\in \N} a_n r^n e^{int}}{\sqrt{1-r^2}}\right|.
\end{align}
Then by noting the equality
\[
\int_\T  f(e^{i \alpha}) P_{r e^{it}} (e^{i\alpha}) \frac{d\alpha}{2\pi} = \sum_{n\in\N} a_n r^n e^{int},
\]
we complete the proof for the desired equality concerning $R_{-}$. The same computation shows that the same is true for $R_{+}$. 
\end{proof}

\begin{remark}\label{rem_H_0}
The identity \eqref{id_poisson} implies that for any $0\le r <1$ and any $t\in \T$,  any $f\in H_0^2(\T)$ and any $\theta \in \T$,  we have 
\[
\left|  \frac{R_{-} (f P_{re^{it}}) (e^{i\theta})}{\sqrt{P_{r e^{it}}  (e^{i \theta})}} \right| = \frac{r}{\sqrt{1-r^2}} \left|  \sum_{n = 1}^\infty a_n r^{n-1} e^{i(n-1)t}\right|   =   \frac{r}{\sqrt{1-r^2}} \left|\int_\T   e^{-i\alpha}f(e^{i \alpha})P_{r e^{it}} (e^{i\alpha}) \frac{d\alpha}{2\pi}\right|.
\]
\end{remark}

Now we can prove Proposition \ref{prop_hankel_g1}.
\begin{proof}[Proof of Proposition \ref{prop_hankel_g1}]
By Remark \ref{rem_H_0}, for any $0\le r <1$,  any $t, \theta \in \T$, any $f\in H_0^2(\T)$ and then any $\lambda > 0$,  we have 
\begin{align*}
\big| [R_{-}(f P_{re^{it}})] (e^{i\theta}) \big| & \le \frac{r}{\sqrt{1-r^2}}  \sqrt{P_{r e^{it}} (e^{i\theta})} \Big(\int_\T | f(e^{i\alpha})|^2 P_{r e^{it}} (e^{i\alpha}) \frac{d\alpha}{2 \pi}\Big)^{1/2} 
\\
& \le  \frac{r}{\sqrt{1-r^2}} \left[ \frac{\lambda}{2} P_{re^{it}} (e^{i\theta}) + \frac{1}{2 \lambda} \int_\T |f(e^{i\alpha})|^2 P_{r e^{it}} (e^{i\alpha}) \frac{d\alpha}{2\pi}\right].
\end{align*}
The above inequality implies that for any positive Radon measure $\mu$ on $\T$, if we set $\varphi = P_r*\mu = \int_\T P_{r e^{it}} d\mu(t)$, then 
\[
\big| [R_{-}(f \varphi)] (e^{i\theta}) \big|
\le  \frac{r}{\sqrt{1-r^2}} \left[ \frac{\lambda}{2} \varphi(e^{i\theta}) + \frac{1}{2 \lambda} \int_\T |f(e^{i\alpha})|^2 \varphi(e^{i\alpha}) \frac{d\alpha}{2\pi}\right].
\]
Then by optimizing the above inequality on $\lambda>0$, we obtain 
\[
\big| [R_{-}(f \varphi)] (e^{i\theta}) \big|
\le  \frac{r}{\sqrt{1-r^2}}  \sqrt{\varphi(e^{i\theta})} \| f\|_{L^2(\T; \varphi)}.
\]
This completes the proof of the inequality \eqref{ineq_c_r}. Finally, the optimality of the constant $r/\sqrt{1-r^2}$ follows by taking $\varphi = P_r$ and $f(e^{i\alpha}) = e^{i\alpha}$. 
\end{proof}


\begin{proof}[Proof of Proposition \ref{prop_pos_coef}]
Let $\mathcal{P}_{H^2}$ denote the set of analytic trigonometric polynormials  $f$ such that $\|f \|_{H^2}\le 1$. Clearly $\mathcal{P}_{H^2}$ is dense in the unit ball of $H^2(\T)$. We have   
\begin{align}\label{hankel_norm}
\begin{split}
\sup_{f \in \mathcal{P}_{H^2}} \| H_\varphi (f)\|_{L^\infty(\T)} & = \sup_{ \sum_{n\in \N} |a_n|^2 \le 1} \sup_{\theta \in \T} \left|   \sum_{n=0}^\infty a_n e^{i n \theta} \sum_{m: m + n \le 0}  \widehat{\varphi}(m) e^{i m \theta} \right| 
\\
& = \sup_{\theta \in \T}  \left( \sum_{n=0}^\infty   \left|   \sum_{m=-\infty}^{-n} \widehat{\varphi}(m) e^{i m \theta}\right|^2 \right)^{1/2}. 
\end{split}
\end{align}
Clearly, if all Fourier coefficients of  $R_{-}(\varphi)$ are positive, then we have 
\[
\| H_\varphi\|_{B(H^2, L^\infty)} = \sup_{f \in \mathcal{P}_{H^2}} \| H_\varphi (f)\|_{L^\infty(\T)} = \left( \sum_{n=0}^\infty   \left(   \sum_{m=-\infty}^{-n} \widehat{\varphi}(m) \right)^2 \right)^{1/2}. 
\]
This completes the proof of the proposition. 
\end{proof}

For proving Proposition \ref{prop_2_inf}, we need the following 
\begin{lemma}\label{lem_HLP}
Assume that $(a_n)_{\ell \ge 0}$ is a sequence in $\C$ with $\sum_{n\in\N} |a_n|^2 \le 1$. Then 
\[
\sum_{n = -\infty}^0 \frac{1}{1 + n^2}   \Big( \sum_{k = n}^0  |a_{k-n}|\Big)^2 <\infty. 
\]
\end{lemma}
\begin{proof}
There exits a constant $C> 0$, such that $\sum_{N}^\infty ( 1 + m^2)^{-1} \le CN^{-1}$ for any integer $N\ge 1$. Therefore, 
\begin{align*}
& \sum_{n = -\infty}^0 \frac{1}{1 + n^2}   \Big( \sum_{k = n}^0  |a_{k-n}|\Big)^2   = \sum_{m=0}^\infty \frac{1}{1 + m^2} \sum_{\ell, \ell' =0}^m | a_\ell| | a_{\ell'}| 
\\
& = \sum_{\ell, \ell' =0}^\infty  | a_\ell| | a_{\ell'}| \sum_{m = \max( \ell, \ell')}^\infty \frac{1}{1 + m^2} \le C \sum_{\ell, \ell' =0}^\infty  \frac{| a_\ell| | a_{\ell'}|}{\max(\ell, \ell')}.
\end{align*}
By Hardy-Littlewood-P\'olya inequality \eqref{HLP-ineq}, we complete the proof of the lemma. 
\end{proof}

\begin{proof}[Proof of Proposition \ref{prop_2_inf}]
Assume first that $H_\varphi: H^2(\T) \longrightarrow L^\infty(\T)$ is a bounded operator.  Then the  equality  \eqref{hankel_norm} implies  that 
\begin{multline*}
\infty > \int_\T   \sum_{n=0}^\infty   \left|   \sum_{m=-\infty}^{-n} \widehat{\varphi}(m) e^{i m \theta}\right|^2   \frac{d \theta}{2\pi} =  \sum_{n=0}^\infty \sum_{m = -\infty}^{-n} | \widehat{\varphi} (m)|^2  = 
\\
 =\sum_{m=-\infty}^0  |\widehat{\varphi}(m)|^2 \sum_{n = 0}^{-m}  1 =  \sum_{m=-\infty}^0 (1 + |m|) |\widehat{\varphi}(m)|^2 \ge \sum_{m = -\infty}^0 (1 + m^2)^{1/2} | \widehat{\varphi}(m)|^2.
\end{multline*}
This implies that $R_{-}(\varphi) \in \BH^{\frac{1}{2}}(\T)$. 

Conversely, assume now that $\varphi$ is a symbol such that 
$
R_{-}(\varphi)\in  \BH^{1}(\T).
$
That is, 
\[
\sum_{m = -\infty}^0 (1 + m^2) | \widehat{\varphi}(m)|^2 <\infty.
\]
Therefore, by \eqref{def_h_f},  for any trigomometric analytic polynomial $f(e^{i\theta}) = \sum_{n\in \N} a_n e^{in \theta}$ such that $\| f\|_{H^2}^2 = \sum_{n\in\N} |a_n|^2 \le 1$, by Cauchy-Schwarz inequality and Lemma \ref{lem_HLP},  we have, 
\begin{align*}
 \| H_\varphi (f)\|_{\BA_1(\T)}  & =  \sum_{k = -\infty}^0  \sum_{n =-\infty}^{k}|  \widehat{\varphi}(n) a_{k-n} |   = \sum_{n=-\infty}^0 |  \widehat{\varphi}(n) | \sum_{k= n}^0 |a_{k-n}| 
\\
& \le \Big( \sum_{n=-\infty}^0 |  \widehat{\varphi}(n) |^2 (1 + n^2) \Big)^{1/2 } \cdot \Big( \sum_{n = -\infty}^0 \frac{1}{1 + n^2}   \big( \sum_{k = n}^0  |a_{k-n}|\big)^2 \Big)^{1/2}<\infty. 
\end{align*}
Thus $H_\varphi$ defines a bounded operator from $H^2(\T)$ to $\BA_1(\T)$.
\end{proof}

\section*{Acknowledgements}
The authors would like to thank Prof. Jeffrey E. Steif for useful discussions. Y. Qiu is supported by grants NSFC Y7116335K1,  NSFC 11801547 and NSFC 11688101 of National Natural Science Foundation of China. Z. Wang is supported by NSFC 11601296.


\end{document}